\documentclass[10pt]{article}

\usepackage{amsmath,amsfonts,amssymb,amsthm,graphicx}
\usepackage{enumerate}
\usepackage{etoolbox}
\usepackage[colorlinks=true,citecolor=blue,pdfpagemode=UseNone,pdfstartview=FitH]{hyperref}
\usepackage{cite}

\usepackage[T2A]{fontenc}
\newenvironment{cyr}{}{}

\allowdisplaybreaks[4]

\newcommand{\noop}[1]{}

\renewcommand{\d}{\,\mathrm{d}}

\newcommand{\N}{\mathbb{N}}

\newcommand{\R}{\mathbb{R}}

\newcommand{\st}{:}

\newcommand{\m}{\mathbf{m}}
\newcommand{\BB}{\mathrm{BB}}
\newcommand{\exc}{\mathrm{e}}

\DeclareMathOperator{\supp}{supp}

\newcounter{cnumber}
\newcommand{\cnew}{\stepcounter{cnumber}c_{\arabic{cnumber}}}
\newcommand{\cold}{c_{\arabic{cnumber}}}
\newcommand{\colder}{c_{\the\numexpr\value{cnumber}-1\relax}}
\newcommand{\coldest}{c_{\the\numexpr\value{cnumber}-2\relax}}
\makeatletter
\@ifundefined{ifmeasuring@}{\newif\ifmeasuring@}{}
\newif\ifcconst@changed
\newcommand{\cconst@aux}[2]{%
  \global\@namedef{cconst@prev@#1}{#2}%
  \global\@namedef{cconst@#1}{#2}%
}
\newcommand{\clabel}[1]{%
  \edef\cconst@tempa{\arabic{cnumber}}%
  \expandafter\xdef\csname cconst@#1\endcsname{\cconst@tempa}%
  \ifmeasuring@\else
    \@ifundefined{cconst@seen@#1}{}%
      {\@latex@warning{Constant label `#1' multiply defined}}%
    \global\@namedef{cconst@seen@#1}{}%
    \@ifundefined{cconst@prev@#1}%
      {\global\cconst@changedtrue}%
      {\expandafter\ifx\csname cconst@prev@#1\endcsname\cconst@tempa\else
         \global\cconst@changedtrue
       \fi}%
    \protected@write\@auxout{}{\string\cconst@aux{#1}{\cconst@tempa}}%
  \fi
}
\DeclareRobustCommand{\cref}[1]{%
  c_{\@ifundefined{cconst@#1}%
    {\@latex@warning{Constant label `#1' undefined}??}%
    {\csname cconst@#1\endcsname}}%
}
\AtEndDocument{%
  \ifcconst@changed
    \@latex@warning@no@line{Constant label(s) may have changed.
      Rerun to get the constants right}%
  \fi
}
\makeatother

\theoremstyle{plain}
\newtheorem{theorem}{Theorem}
\newtheorem{corollary}[theorem]{Corollary}
\newtheorem{lemma}[theorem]{Lemma}
\newtheorem{proposition}[theorem]{Proposition}

\theoremstyle{definition}

\theoremstyle{remark}
\newtheorem{remark}[theorem]{Remark}

\newcommand{\tit}%
  {The universal measure of probabilistically nonrandom objects}
\newcommand{\abstr}%
  {A finite object is called probabilistically random
  if it is an algorithmically random element of a finite set
  whose Kolmogorov complexity is small
  as compared with the complexity of the object.
  This paper studies the prevalence of probabilistically nonrandom objects
  (those that are not probabilistically random)
  as gauged by the universal measure
  and without assuming any structure on the objects.
  The main results imply that the universal measure of probabilistically nonrandom objects
  is vanishingly small among objects of high complexity
  and establish the dependence of their universal measure
  on the required degree of probabilistic randomness.}

\begin{document}
\title{\tit}
\author{Vladimir Vovk}
\date{December 3, 2026}
\maketitle
\begin{abstract}
  \smallskip
  \abstr

  The version of this paper at \url{http://gtfp.net} (Working Paper 71)
  is updated most often.
\end{abstract}

\section{Introduction}
\label{sec:intro}

This paper is a natural complement to \cite{Vovk:arXiv2608-local};
both papers develop an approach to algorithmic statistics
(see, e.g.,
\cite{Gacs/etal:2001,Vereshchagin/Vitanyi:2004,Vereshchagin/Shen:2015,Vereshchagin/Shen:2017})
inspired by Kolmogorov's original definition
of stochastic objects \cite[Note~5.12]{Semenov/etal:2024-full}.
The standard definitions in algorithmic statistics
assume a given structure on the space of finite objects under consideration,
such as the binary strings $\{0,1\}^*$ being partitioned
into the sets $\{0,1\}^n$,
with the lengths of strings playing an important role
in the statements of the main results.
Such structures are to some degree arbitrary;
e.g., we may consider other increasing sequences of stopping times,
different from the constant stopping times $0,1,2,\dots$.
Besides, structure-free settings are attractive due to their simplicity.

The previous paper \cite{Vovk:arXiv2608-local}
discusses a structure-free version of standard results
on the universal measure of $(\alpha,\beta)$-stochastic objects,
where an object is \emph{$(\alpha,\beta)$-stochastic}
if it has a model $\Omega$ of complexity at most $\alpha$
in which its randomness deficiency is at most $\beta$.
(Formal definitions will be given later in the paper.)
While the standard results have been interpreted as saying that
``the probability of obtaining a non-stochastic object in a random process is negligible''
\cite[Sect.~4.6]{Vereshchagin/Shen:2017},
results of \cite{Vovk:arXiv2608-local} say
that $(\alpha,\beta)$-nonstochastic objects are abundant:
their universal measure tends to 0 as $\alpha\to\infty$
($\beta$ plays hardly any role),
but extremely slowly.

A fundamental characteristic of a finite object $\omega$ in algorithmic statistics
is its best fit function $\beta_{\omega}$
(to use the terminology of \cite[(II.7)]{Vereshchagin/Vitanyi:2004});
for each $\alpha\ge0$,
$\beta_{\omega}(\alpha)$ is the smallest randomness deficiency of $\omega$
over its models $\Omega$ of complexity at most $\alpha$.
Each function $\beta_{\omega}$ is monotonically decreasing to $0$,
to within an additive constant (it may take small negative values).

A popular alternative to the notion of the best fit function
is that of the structure function for a finite object $\omega$.
Historically, structure functions appeared in print
even before best fit functions \cite{Kolmogorov:1974},
although there is evidence \cite[Sect.~3.3, p.~858]{Cover/etal:1989}
that Kolmogorov defined best fit functions
already in 1973 in his talk
at an information theory conference in Tallinn;
it is definitely true that he defined structure functions in that talk
\cite[the photo in Sect.~2]{Semenov/etal:2024-full}.
(Cover et al.\ \cite[p.~858]{Cover/etal:1989} say that Kolmogorov
``proposed a variant of'' best fit functions,
so it is also possible that those authors consider structure functions
to be a variant of best fit functions.)

At the Tallinn conference Kolmogorov posed the problem
of showing that the best fit functions (or structure functions, or both)
can take approximately arbitrary shapes
subject to the restriction of being monotonically decreasing to 0
\cite[p.~858]{Cover/etal:1989}.
This was shown to be indeed the case
by V'yugin in 1987 \cite[Corollary]{Vyugin:1987} in a restricted setting
and by Vereshchagin and Vit\'anyi in 2004 \cite[Corollary~IV.9]{Vereshchagin/Vitanyi:2004}
in general;
see also \cite[Theorem~17.5]{Vereshchagin/Shen:2015}.

The results about possible shapes of best fit functions are qualitative
and do not indicate how common different shapes are.
In this paper we will be interested in how common they are
under the universal measure in the structure-free setting.
It turns out that under the universal measure the most common shapes are those
characteristic of probabilistically random objects,
which were introduced in print by Kolmogorov in 1974 \cite{Kolmogorov:1974}
in the language of structure functions.
In terms of best fit functions,
an object $\omega$ is probabilistically random
(\begin{cyr}вероятностно случаен\end{cyr})
if $\beta_{\omega}(\alpha)$ becomes small,
which we will interpret as $O(1)$,
already for a relatively small value
(\begin{cyr}сравнительно небольшом значении\end{cyr})
of $\alpha$.
Since best fit functions are monotonically decreasing,
$\beta_{\omega}$ will stay small to the right of $\alpha$ as well.
In our principal results,
such as \eqref{eq:set-1}--\eqref{eq:set-4},
we will interpret $\alpha$ being ``relatively small'' as
$\alpha=\theta K(\omega)$ (for a small $\theta>0$),
$\alpha=K(\omega)^{\theta}$ (for $\theta\in(0,1)$),
$\alpha=\log^{\theta}K(\omega)$ (for $\theta>1$),
and $\alpha=\theta\log K(\omega)$ (also for $\theta>1$),
respectively,
where $K(\omega)$ is the complexity of $\omega$.
For the objects $\omega$ of complexity $K(\omega)>a$,
the universal measure of probabilistically nonrandom objects
will shrink, as $a\to\infty$, to 0 exponentially fast,
stretched-exponentially fast,
superpolynomially fast,
and polynomially fast, respectively.

The interpretation of $\beta_{\omega}(\alpha)$ being small
as $\beta_{\omega}(\alpha)=O(1)$ is natural
in view of the standard statistical and scientific practices:
using the thresholds $5\%$ and $1\%$
for statistical significance and high statistical significance,
respectively,
using the more demanding threshold $3\times10^{-3}$
for ``evidence of a particle'',
and using the even more demanding threshold of ``$5\sigma$'',
or approximately $3\times10^{-7}$,
for discoveries in particle physics \cite{Lamb:2012}.
On the binary logarithmic scale
(as used in the algorithmic theory of randomness),
these thresholds would correspond to,
approximately, 4.3, 6.6, 8.4, and 21.7.

In the standard structured setting,
common shapes of best fit functions can be deduced
from the preponderance, under the universal measure, of stochastic objects,
which can be defined as the objects for which
the graphs of their best fit functions pass close to the origin.
For stochastic objects the best fit function quickly drops almost to zero
and then stays close to zero.
Therefore, the common objects under the universal measure are probabilistically random.
The situation in structure-free algorithmic statistics is radically different,
since now nonstochastic objects are abundant \cite[Theorem~1]{Vovk:arXiv2608-local}.
Nevertheless, the results of this paper still say
that objects that are not probabilistically random are a negligible minority
among objects of high complexity under the universal measure.

There is a contrast, or even a tension, between the statement of \cite{Vovk:arXiv2608-local}
that nonstochastic objects are abundant
and the statement of this paper that there are few probabilistically nonrandom objects.
As interpreted here,
probabilistically random objects are a special case
of $(\alpha,\beta)$-stochastic objects $\omega$
corresponding to $\beta=O(1)$ and $\alpha=f(K(\omega))$ for a slowly growing $f$.
In our current setting we concentrate on the objects $\omega$ satisfying $K(\omega)>a$
for large $a$
(in the next section we will see, in Proposition~\ref{prop:many},
that overall there are many such $\omega$).
This restriction is not what causes the contrast,
since \cite{Vovk:arXiv2608-local} also contains results involving it (such as \cite[Theorem~6]{Vovk:arXiv2608-local}).
The main difference responsible for the contrast
is that in this paper we are interested in $\alpha$ as a function of $K(\omega)$,
while in \cite{Vovk:arXiv2608-local} $\alpha$ and $\beta$ are free parameters.

There are several possible settings for our topic.
First, as already mentioned,
Kolmogorov's structure function is a popular alternative
to the best fit function of a finite object,
and another alternative is the MDL function
(see, e.g., \cite[(II.8) and (II.10)]{Vereshchagin/Vitanyi:2004}).
Second, in our definitions
we can use plain Kolmogorov complexity $C$ or prefix complexity $K$.
For simplicity and concreteness,
we concentrate on best fit functions and prefix complexity.
Because of their historical significance, mentioned earlier,
structure functions will be briefly discussed in Appendix~\ref{app:structure}.

Let $\N:=\{1,2,\dots\}$ and $\N_0:=\{0,1,\dots\}$,
$\log$ be the binary logarithm,
$K$ be prefix complexity,
and $\m$ be the universal semimeasure on $\N$,
as defined in, e.g., \cite{Shen/etal:2017book}.
For $A\subseteq\N$, define $\m(A):=\sum_{\omega\in A}\m(\omega)$;
with this extension, we also refer to $\m$ as the \emph{universal measure}
(since formally it becomes a measure).
Since we do not consider plain Kolmogorov complexity $C$,
the words ``describe'' and ``description'' are used in the self-delimiting sense.
Let $\tau:\N_0\to(0,\infty)$ be the tail of $\m$:
$\tau(t):=\m(\{n\st n>t\})$.
The words ``increasing'' and ``decreasing'' will be understood in the wide sense
allowing intervals of constancy;
we will drop ``monotonically'' from now on.
The symbols $O$, $o$, and $\Theta$ will be used only in informal explanations
and, occasionally, in proofs where their meaning is unambiguous.

\section{Finite sets as models}
\label{sec:set}

Let us refer to elements of $\N$ as \emph{objects}
(to emphasize that our results and discussions are applicable
to other spaces of finite objects, such as $\{0,1\}^*$,
in computable bijection with $\N$).
A finite set $\Omega\subseteq\N$ containing an object $\omega$
is a \emph{model} for $\omega$;
we are interested in simple models,
in the sense of $K(\Omega)$ being small,
for which the \emph{randomness deficiency}
\begin{equation}\label{eq:deficiency}
  d(\omega\mid\Omega)
  :=
  \log\left|\Omega\right|
  -
  K(\omega\mid\Omega)
\end{equation}
is small.
For each $\omega\in\N$,
the \emph{best fit function} $\beta_{\omega}:[0,\infty)\to\R\cup\{\infty\}$
is defined by
\[
  \beta_{\omega}(\alpha)
  :=
  \min_{\Omega:K(\Omega)\le\alpha}
  d(\omega\mid\Omega),
\]
$\Omega$ ranging over the models for $\omega$;
as usual, $\min\emptyset:=\infty$.
The effective range of $\alpha$ is $[0,K(\omega)+O(1)]$,
since $\beta_{\omega}(K(\omega)+O(1))=O(1)$
(witnessed by $\Omega:=\{\omega\}$,
whose complexity exceeds $K(\omega)$ by at most an additive constant).
It is clear that the function $\beta_{\omega}$ is decreasing,
piecewise constant, and right-continuous
(namely, it is constant on each interval $[n,n+1)$, $n\in\N_0$).

We are interested in possible shapes for $\beta_{\omega}$
for complex $\omega$
(although in this paper we will only deal with primitive mostly flat shapes).
First let us see how many such $\omega$ there are overall.

\begin{proposition}\label{prop:many}
  For some constant $c>1$ and for all $a\in\N_0$,
  \[
    \tau(a)/c
    \le
    \m
    \left(
      K>a
    \right)
    \le
    c\tau(a).
  \]
\end{proposition}

The double inequality in Proposition~\ref{prop:many}
is a shorthand, standard in probability theory, for
\[
  \tau(a)/c
  \le
  \m
  \left(
    \{\omega\in\N\st K(\omega)>a\}
  \right)
  \le
  c\tau(a).
\]
This very simple estimate,
proved, e.g., in the companion paper \cite[Lemma~4]{Vovk:arXiv2608-local},
shows that there are quite a few $\omega$ with $K(\omega)>a$.
Their universal measure tends to 0, of course,
as $a\to\infty$, but it does so very slowly
(see, e.g., \cite[Lemma~2]{Vovk:arXiv2608-local}).
Now let us see that the vast majority of those objects
are probabilistically random as measured by the universal measure.

First we state two general results,
Theorems~\ref{thm:main} and~\ref{thm:opt} (proved in Sect.~\ref{sec:set-proofs});
they look somewhat complicated,
but Corollary~\ref{cor:set-g} and, especially,
inequalities \eqref{eq:set-1}--\eqref{eq:set-4} stated afterwards
are simpler and more intuitive,
and those can be regarded as the main results of this section.
In these theorems and later on
we let $c_1,c_2,\dots$ stand for positive constants,
perhaps depending on values or functions regarded as constant;
we also allow dependence of $c_j$ on $c_i$ for $i<j$
and of a constant in a statement on the constants introduced in its proof
(other dependences will be pointed out explicitly).

\begin{theorem}\label{thm:main}
  Let $g$ be a computable function $g:\N\to[0,\infty)$
  with $g(k)\le k$ for all $k\ge c_g$.
  There is a positive constant $\cnew$\clabel{c:main} such that,
  for all $a\in\N_0$,
  \begin{equation}\label{eq:main}
    \m
    \left(
      K>a,
      \beta(g(K))>\cref{c:main}
    \right)
    \le
    \cref{c:main}
    \sum_{k=a+1}^{\infty}
    2^{-g(k)}.
  \end{equation}
\end{theorem}

Again, \eqref{eq:main} is a shorthand for
\begin{equation*}
  \m
  \left(
    \{
      \omega\in\N\st K(\omega)>a,
      \beta_{\omega}(g(K(\omega)))>\cref{c:main}
    \}
  \right)
  \le
  \cref{c:main}
  \sum_{k=a+1}^{\infty}
  2^{-g(k)}.
\end{equation*}
Such shorthands will be used constantly
in the rest of the paper.
In the proof of Theorem~\ref{thm:main} in Sect.~\ref{sec:set-proofs}
we will in fact deduce \eqref{eq:main} from the more elegant
\begin{equation}\label{eq:main-1}
  \forall k\in\N:
  \quad
  \m
  \left(
    K=k,
    \beta(g(k))>\cref{c:main}
  \right)
  \le
  \cref{c:main}
  2^{-g(k)}.
\end{equation}

\begin{theorem}\label{thm:opt}
  Let $g:\N\to[0,\infty)$ be a computable increasing
  unbounded function.
  There is a positive integer constant $\cnew$\clabel{c:opt} such that,
  for
  \begin{equation}\label{eq:opt-1}
    m_g(a)
    :=
    \min
    \left\{
      m\in\N_0
      \st
      g
      \left(
        \max(m,a)+\cref{c:opt}
      \right)
      \le
      m
    \right\}
  \end{equation}
  and for all $a\in\N_0$,
  \begin{equation}\label{eq:opt-2}
    \m
    \left(
      K>a,
      \beta(g(K))
      =
      \infty
    \right)
    \ge
    2^{-m_g(a)-\cref{c:opt}}.
  \end{equation}
\end{theorem}

\noindent
As usual, $\min\emptyset$ in \eqref{eq:opt-1} is understood to be $\infty$,
in which case the right-hand side of \eqref{eq:opt-2}
is understood to be $0$
(and so the inequality \eqref{eq:opt-2} holds trivially
in this case).
However, in this paper we will be interested in $g$
satisfying $k-g(k)\to\infty$ as $k\to\infty$,
and so we will have $m_g(a)<\infty$ for all $a$
(see, e.g., the first inequality in~\eqref{eq:regular} below).
The requirement that $g$ be unbounded cannot be dropped:
for a bounded $g$, the right-hand side of~\eqref{eq:opt-2}
would be bounded away from $0$,
whereas $\m(K>a)\to0$ as $a\to\infty$.

Theorem~\ref{thm:opt} is close to being a converse
of Theorem~\ref{thm:main}.
Indeed, if $g(a+\cref{c:opt})\le a$,
the minimum in~\eqref{eq:opt-1} is attained at
$m:=\lceil g(a+\cref{c:opt})\rceil$,
and so the right-hand side of~\eqref{eq:opt-2}
is $2^{-g(a+O(1))-O(1)}$,
which is essentially the first (and largest) term of the series in~\eqref{eq:main}.
Informally, we can say that, under some regularity condition on $g$,
\[
  \m
  \left(
    K>a,
    \beta(g(K))=\infty
  \right)
  \approx
  2^{-g(a)}.
\]
This observation is formalized in the following statement.

\begin{corollary}\label{cor:set-g}
  Let $g:[1,\infty)\to[0,\infty)$ be a computable increasing function
  that is twice differentiable and satisfies,
  for some $\epsilon>0$ and all sufficiently large $x$,
  \begin{equation}\label{eq:regular}
    0<g'(x)\le1-\epsilon
    \quad\text{and}\quad
    \left(1/g'\right)'(x)\le(1-\epsilon)\ln2.
  \end{equation}
  There is a positive constant $c$ such that, from some $a$ on,
  \begin{multline}\label{eq:set-g}
    2^{-g(a)-c}
    \le
    \m
    \left(
      K>a,
      \beta(g(K))=\infty
    \right)\\
    \le
    \m
    \left(
      K>a,
      \beta(g(K))>c
    \right)
    \le
    c 2^{-g(a)} / g'(a).
  \end{multline}
\end{corollary}

\noindent
The first requirement in~\eqref{eq:regular} says that $g$ should never grow too fast,
and it is natural since we are interested in the bounds $g(K)$
on the complexity of our models that are significantly smaller than the complexity of the object, $K$.
The second requirement says that the growth of $g$ should be steady.

Corollary~\ref{cor:set-g} determines both universal measures in~\eqref{eq:set-g}
to within a factor of order $1/g'(a)$,
and therefore, to within a constant factor when $g'$ is bounded away from $0$.

\begin{proof}[Proof of Corollary~\ref{cor:set-g}]
  The upper bound follows from Theorem~\ref{thm:main},
  whose condition $g(k)\le k$ for large $k$
  follows from $g'\le1-\epsilon$ in~\eqref{eq:regular}.
  It gives, for large $a$,
  \[
    \sum_{k=a+1}^{\infty}2^{-g(k)}
    \le
    \int_a^{\infty}2^{-g(x)}\d x
    \le
    \frac{1}{\epsilon\ln2}
    \left[
      -\frac{2^{-g}}{g'}
    \right]_a^{\infty}
    \le
    \frac{2^{-g(a)}}{\epsilon(\ln2)g'(a)},
  \]
  where the upper bound on $2^{-g}$ used in the second ``$\le$''
  is the inequality between the extreme terms of
  \[
    \left(\frac{2^{-g}}{g'}\right)'
    =
    2^{-g}
    \left(
      -\ln2+\left(\frac{1}{g'}\right)'
    \right)
    \le
    -\epsilon(\ln2)2^{-g};
  \]
  the last inequality follows, for large $x$,
  from the second inequality in~\eqref{eq:regular}.

  For the lower bound,
  we apply Theorem~\ref{thm:opt}
  (notice that $g$ is unbounded,
  since $1/g'(x)=O(x)$ by the second inequality in~\eqref{eq:regular}).
  Let $m:=\lceil g(a+\cref{c:opt})\rceil$.
  By the first inequality in~\eqref{eq:regular}, $m\le a$ for large $a$.
  We can see that $m$ satisfies the defining inequality in~\eqref{eq:opt-1},
  and so $m_g(a)\le m\le g(a)+\cref{c:opt}+1$.
  It remains to apply~\eqref{eq:opt-2}.
\end{proof}

The inequality~\eqref{eq:set-g} implies
\begin{multline}\label{eq:set-g-bis}
  2^{-g(a)-c}
  \le
  \m
  \left(
    K>a,
    \beta(g(K))=\infty
  \right)\\
  \le
  \m
  \left(
    K>a,
    \beta(g(K))>c
  \right)
  \le
  2^{-g(a)+\log a+c}
\end{multline}
for large $a$.
To check~\eqref{eq:set-g-bis},
integrate the second inequality in~\eqref{eq:regular},
obtaining $1/g'(x)\le(1-\epsilon)(\ln2)x+O(1)$;
since $(1-\epsilon)\ln2<1$, this gives $1/g'(a)\le a$ for large $a$.

Now let us specialize Corollary~\ref{cor:set-g} further.
First, taking $g(x):=\theta x$ for a computable number $\theta\in(0,1)$,
we can see that there is a positive constant $c$ such that,
for all $a\in\N_0$,
\begin{equation}\label{eq:set-1}
  2^{-\theta a-c}
  \le
  \m
  \left(
    K>a,
    \beta(\theta K)=\infty
  \right)
  \le
  \m
  \left(
    K>a,
    \beta(\theta K)>c
  \right)
  \le
  2^{-\theta a+c}.
\end{equation}
(Corollary~\ref{cor:set-g} says ``from some $a$ on'',
but the first few $a$s can be absorbed into the constant $c$.)
We can see that the rate of decay of the universal measure
in \eqref{eq:set-1} is $2^{-\theta a+O(1)}$.
Since \eqref{eq:set-1} concerns the universal measure of the $(\alpha,\beta)$-nonstochastic objects $\omega$
with $\beta=c$ and $\alpha:=\theta K(\omega)>\theta a$,
the exponential rate of decay
is reminiscent of the standard results in algorithmic statistics
containing $2^{-\alpha}$ as the leading term
(see, e.g., \cite[Corollary and Theorem~3]{Vyugin:1985}
and \cite[Propositions~5 and~22]{Vereshchagin/Shen:2017}).

Now let us set $g(x):=x^{\theta}$ for a computable $\theta\in(0,1)$;
the conditions in \eqref{eq:regular} are easily checked.
Then there is a positive constant $c$ such that,
for all $a\in\N_0$,
\begin{multline}\label{eq:set-2}
  2^{-a^{\theta}-c}
  \le
  \m
  \left(
    K>a,
    \beta\left(K^{\theta}\right)=\infty
  \right)\\
  \le
  \m
  \left(
    K>a,
    \beta\left(K^{\theta}\right)>c
  \right)
  \le
  2^{-a^{\theta}+(1-\theta)\log(a+2)+c}
\end{multline}
(which is slightly stronger than~\eqref{eq:set-g-bis}).
Now the exponential decay of~\eqref{eq:set-1}
becomes stretched exponential.
The two bounds in~\eqref{eq:set-2} are in close agreement,
the universal measure being $2^{-a^{\theta}+O(\log a)}$.

Next set $g(x):=\log^{\theta}x$ for a computable $\theta>1$.
As \eqref{eq:regular} is again easy to check,
there is a positive constant $c$ such that,
from some $a\in\N_0$ on,
\begin{multline}\label{eq:set-3}
  2^{-\log^{\theta}a-c}
  \le
  \m
  \left(
    K>a,
    \beta
    \left(
      \log^{\theta}K
    \right)
    =
    \infty
  \right)\\
  \le
  \m
  \left(
    K>a,
    \beta
    \left(
      \log^{\theta}K
    \right)
    >
    c
  \right)
  \le
  2^{-\log^{\theta}a+\log a}
\end{multline}
(cf.\ \eqref{eq:set-g-bis}).
Now we have a slower rate of decay
of the universal measure of probabilistically nonrandom objects,
namely superpolynomial decay.
The two bounds in \eqref{eq:set-3} agree up to a factor of $a^{O(1)}$,
the universal measure being $2^{-\log^{\theta}a+O(\log a)}$.

Finally, set $g(x):=\theta\log x$ for a computable $\theta>1$,
where the condition $\theta>1$ is needed to satisfy
the second inequality in~\eqref{eq:regular}.
Now there is a positive constant $c$ such that,
for all $a\in\N_0$,
\begin{multline}\label{eq:set-4}
  (a+2)^{-\theta}/c
  \le
  \m
  \left(
    K>a,
    \beta(\theta\log K)=\infty
  \right)\\
  \le
  \m
  \left(
    K>a,
    \beta(\theta\log K)>c
  \right)
  \le
  c
  (a+2)^{1-\theta}.
\end{multline}
The rate of decay in $a$ in~\eqref{eq:set-4} is even slower,
namely polynomial.
The two bounds in it differ by a factor of order~$a$.

The upper bounds of \eqref{eq:set-1}--\eqref{eq:set-4}
assert that there are few probabilistically nonrandom objects
as gauged by the universal measure;
e.g., \eqref{eq:set-3}
implies that the probabilistically nonrandom objects are superpolynomially few
and~\eqref{eq:set-4}
implies that the probabilistically nonrandom objects are polynomially few.
These bounds trade off the strictness of the requirement of probabilistic randomness
against the preponderance of probabilistically random objects.

\begin{remark}
  Let me discuss a very speculative interpretation
  of this paper's results related
  to Levin's and V'yugin's interpretations of the universal semimeasure.
  The central principle in the applications of probability,
  first of all statistical hypothesis testing,
  is that, in a random experiment modelled by a probability measure,
  we do not expect to observe a pre-specified event of small probability;
  the smaller the probability, the stronger the expectation.
  This principle goes back to Jacob Bernoulli \cite[Part~IV]{Bernoulli:1713}
  and was stated by Kolmogorov as Principle~B
  in his fundamental monograph \cite[Sect.~I.2]{Kolmogorov:1933}.
  This principle and its variations are known as Cournot's principle,
  after Cournot \cite[Sect.~43]{Cournot:1843},
  who considered it to be the only bridge between probability theory and reality
  (see, e.g., \cite[Sect.~2.2]{Shafer/Vovk:2006SS}
  and \cite[Sect.~2.2]{Shafer:2007-short}).
  Similarly, we can consider the principle that,
  in any physical process producing a single outcome,
  we do not expect to observe a pre-specified event of small universal semimeasure,
  with the strength of the expectation similar to that in Cournot's principle
  (for an event of the same probability).
  Let us call it the \emph{universal Cournot principle}.
  Variants of this principle were used
  by Levin \cite{Levin:1984} and V'yugin \cite[end of paper]{Vyugin:1985};
  Levin connected this principle (which he did not state explicitly)
  with his Independence Postulate \cite[Note~3 and Lemma~3]{Levin:1984}.
  If the universal Cournot principle is accepted,
  Corollary~\ref{cor:set-g} and \eqref{eq:set-1}--\eqref{eq:set-4}
  can be interpreted as asserting the probabilistic randomness
  of the complex objects produced by physical processes.
\end{remark}

\section{Proofs of Theorems~\ref{thm:main} and~\ref{thm:opt}}
\label{sec:set-proofs}

\begin{proof}[Proof of Theorem~\ref{thm:main}]
  This proof uses standard techniques and merely applies them in a new way.
  A key role is played by the sets
  $D_k:=\{\omega\in\N\st K(\omega)\le k\}$.
  It is well known that
  \begin{equation}\label{eq:Chaitin}
    \left|D_k\right|\le2^{k-K(k)+O(1)};
  \end{equation}
  see, e.g., \cite[Theorem~64]{Shen/etal:2017book}.
  We will use the technique of ``bounded complexity lists''
  (see, e.g., \cite{Gacs/etal:2001} and \cite[Sect.~4]{Vereshchagin/Shen:2017}),
  and we start by fixing a jointly computable way
  of enumerating all elements of $D_k$ given $k$.
  Without loss of generality we assume that $g$ only takes integer values,
  $g:\N\to\N_0$.
  Without loss of generality we also assume $a\ge c_g$;
  indeed, $a<c_g$ can be absorbed
  by the constant $\cref{c:main}$ in~\eqref{eq:main}.

  Set $k:=K(\omega)$ for a given object $\omega\in\N$.
  We will split $D_k$ in the order in which it is enumerated
  into blocks of equal size $2^b$
  and a remainder of size less than $2^b$
  taking $b:=k-g(k)$
  (here we are using the condition $g(k)\le k$).
  As $\Omega$ we take the block containing $\omega$
  (assuming it exists);
  let us check that $K(\Omega)\le g(k)+O(1)$.
  Indeed, to describe $\Omega$,
  we first describe $k$, which requires $K(k)$ bits
  but also determines $g(k)$ and $b$;
  then the ordinal number of $\Omega$ as a block
  does not exceed
  \begin{equation}\label{eq:exponent}
    2^{k-K(k)+\cnew\clabel{c:count}}/2^b
    =
    2^{k-K(k)+\cold}/2^{k-g(k)}
    =
    2^{g(k)-K(k)+\cold},
  \end{equation}
  which requires at most $g(k)-K(k)+\cnew$ bits given a shortest description of $k$.
  Overall, describing $\Omega$ requires
  at most $g(k)+\cnew\clabel{c:Omega}$ bits.
  (If the last exponent in~\eqref{eq:exponent} is negative,
  there are no blocks and $\omega$ is guaranteed to be in the remainder,
  which case is also covered by the argument below.)

  Let us check that $\omega$ is a random element of $\Omega$:
  \begin{multline}\label{eq:d}
    d(\omega\mid\Omega)
    =
    \log\left|\Omega\right|
    -
    K(\omega\mid\Omega)
    \le
    \log\left|\Omega\right|
    -
    K(\omega) + K(\Omega)
    +
    \cnew\\
    \le
    k-g(k)
    -
    k+g(k)
    +
    \cnew
    =
    \cold.
  \end{multline}

  Now we can deduce essentially the statement of the theorem:
  the $\omega$ satisfying $K(\omega)=k>a$ and
  $\beta_{\omega}(g(k)+\cref{c:Omega})>\cold$
  are among the remainder and so their overall universal measure is less than
  $2^b 2^{-k+\cnew}=2^{-g(k)+\cold}$.
  This gives essentially \eqref{eq:main-1}, and to get \eqref{eq:main}
  it remains to sum over all such $k$.

  The argument so far gives us an extra additive constant
  in~\eqref{eq:main-1}
  ($g(k)+\cref{c:Omega}$ in place of $g(k)$)
  and so in the statement of Theorem~\ref{thm:main}.
  It can be eliminated by increasing $b$ by an additive constant:
  using blocks of size $2^b$ with $b:=k-g(k)+\cnew$,
  $\cold$ being large as compared with $\cref{c:Omega}$,
  divides their number by $2^{\cold}$,
  so that describing $\Omega$ now requires at most $g(k)$ bits,
  while the upper bounds on $d(\omega\mid\Omega)$ and the universal measure of the remainder
  grow by at most an additive constant and a constant factor, respectively,
  and those are absorbed by $\cref{c:main}$.
  (And if there are no blocks, we have no $\Omega$ to describe.)
\end{proof}

The proof of Theorem~\ref{thm:opt} will use the busy-beaver-type function
\begin{equation}\label{eq:B}
  \BB(k):=\max\{n\in\N\st K(n)\le k\}
\end{equation}
(with $\max\emptyset:=0$, as usual);
this is an increasing function that grows extremely quickly.
The main component of the proof of Theorem~\ref{thm:opt}
is the following lemma.

\begin{lemma}\label{lem:max}
  There exists a constant $\cnew$\clabel{c:max} such that,
  for every object $\omega\in\N$ and every $\alpha\ge0$,
  we have $\beta_{\omega}(\alpha)=\infty$ whenever $\omega>\BB(\alpha+\cref{c:max})$.
\end{lemma}

\begin{proof}
  Let $\cref{c:max}$ be such that
  $K(\max\Omega)\le K(\Omega)+\cref{c:max}$ for all $\Omega$.
  The statement of the lemma follows from there being no models
  $\Omega$ for $\omega>\BB(\alpha+\cref{c:max})$ of complexity $K(\Omega)\le\alpha$:
  indeed, for any model $\Omega$ for $\omega$,
  \begin{multline*}
    \omega>\BB(\alpha+\cref{c:max})
    \Longrightarrow
    \max\Omega>\BB(\alpha+\cref{c:max})\\
    \Longrightarrow
    K(\max\Omega)>\alpha+\cref{c:max}
    \Longrightarrow
    K(\Omega)>\alpha.
    \qedhere
  \end{multline*}
\end{proof}

The intuition behind Lemma~\ref{lem:max} is that a model for $\omega$
must in particular name a number at least as large as $\omega$,
so no object has models cheaper than the least complexity of a number
equal to or exceeding it.
The idea of the proof of Theorem~\ref{thm:opt}
is to extend a busy-beaver number by random bits,
thus obtaining many numbers at a controlled complexity.

We will need two more lemmas for the proof
(they will also be used in the proof of Theorem~\ref{thm:set-g-window}).
The first one gives some rigorous variations
on the heuristic approximate equality $\m(K=k)\approx\m(k)$
(cf.\ Proposition~\ref{prop:many}).
This elegant relation is not literally true:
e.g., we can pad descriptions
to make $K(\omega)$ even for all objects $\omega$;
however, it is very close to being true.
Namely, the following lemma implies that
\begin{equation}\label{eq:elegant}
  \m(k)/\cref{c:elegant}
  \le
  \m(k<K\le k+\cnew)
  \le
  \cnew\clabel{c:elegant}\m(k)
\end{equation}
and gives a conditional version of the lower bound.

\begin{lemma}\label{lem:window}
  There are constants $d_0,k_0\in\N$ and $\cnew$\clabel{c:window} such that,
  for all binary strings $y$ and all $k\ge k_0$,
  \begin{equation}\label{eq:window-lower}
    \sum_{z\in\N:\,k<K(z\mid y)\le k+d_0}
    2^{-K(z\mid y)}
    \ge
    2^{-K(k\mid y)-\cref{c:window}}.
  \end{equation}
  Specializing this to the empty $y$,
  and increasing $d_0$ and $\cref{c:window}$ if necessary,
  gives $\m(k<K\le k+d_0)\ge2^{-K(k)-\cref{c:window}}$ for all $k\ge k_0$.
  In the other direction, $\m(K=k)\le2^{-K(k)+\cref{c:window}}$ for all $k\in\N$.
\end{lemma}

\begin{proof}
  Let $\cnew$\clabel{c:e} be a constant to be chosen later,
  and set $\ell:=k-K(k\mid y)+\cref{c:e}$;
  it is clear that $\ell\ge1$ for $k\ge k_0$, $k_0$ being a constant.
  For $u\in\{0,1,\dots,2^{\ell}-1\}$, set $z_u:=2^k(2u+1)$;
  both $k$ and $u$ can be computed from $z_u$.
  We will establish~\eqref{eq:window-lower}
  with the sum over $z_u$.

  First we estimate $K(z_u\mid y)$.
  Given $y$, we can describe $z_u$
  by a shortest description of $k$ given $y$
  followed by $u$ written in exactly $\ell$ binary digits;
  since $\ell$ is computable from $k$ and the length $K(k\mid y)$ of that description,
  \[
    K(z_u\mid y)
    \le
    K(k\mid y)+\ell+\cnew
    \le
    k+d_0,
  \]
  where $\cold$ and $d_0$
  (the one used in \eqref{eq:window-lower})
  depend on $\cref{c:e}$.
  On the other hand, by the complexity of a pair theorem
  (see, e.g., \cite[Theorem~67]{Shen/etal:2017book})
  relativized to $y$,
  \[
    K(z_u\mid y)
    \ge
    K(k,u\mid y)-O(1)
    =
    K(k\mid y)+K(u\mid k,K(k\mid y),y)-O(1),
  \]
  where the $O(1)$ does not depend on $\cref{c:e}$.
  Since fewer than $2^{\ell-1}$ numbers $u<2^{\ell}$
  satisfy $K(u\mid k,K(k\mid y),y)<\ell-1$,
  at least $2^{\ell-1}$ of them satisfy
  \[
    K(z_u\mid y)
    \ge
    K(k\mid y)+\ell-1-O(1)
    =
    k+\cref{c:e}-O(1)
    >
    k,
  \]
  where the last inequality holds for a large enough $\cref{c:e}$.
  Therefore,
  \[
    \sum_{z:\,k<K(z\mid y)\le k+d_0}
    2^{-K(z\mid y)}
    \ge
    2^{\ell-1}2^{-k-d_0}
    =
    2^{-K(k\mid y)+\cref{c:e}-1-d_0},
  \]
  which gives \eqref{eq:window-lower}.

  For the last statement in the lemma,
  \eqref{eq:Chaitin} gives
  \[
    \m(K=k)
    \le
    \sum_{\omega:K(\omega)=k}
    2^{-k+O(1)}
    \le
    2^{-k+O(1)}
    \left|D_k\right|
    \le
    2^{-K(k)+O(1)}.
    \qedhere
  \]
\end{proof}

The following lemma combines,
in a non-trivial way,
\eqref{eq:elegant} and the bounds
\[
  2^{-m}/\cnew
  \le
  \tau(\BB(m))
  \le
  \cold
  2^{-m+K(m)}
\]
(the lower bound being trivial
and the upper bound following from Andreev's \cite{Andreev:2017} results).
The lemma, however, gives only a lower bound.

\begin{lemma}\label{lem:BB-window}
  Let $d_0$ and $k_0$ be as in Lemma~\ref{lem:window}.
  There are constants $d\ge d_0$ and $\cnew$\clabel{c:BB-window} such that,
  for all integers $m\ge K(1)$ and all $b\ge m+k_0$,
  \begin{equation}\label{eq:BB-window}
    \m
    \left(
      \left\{
        \omega\in\N
        \st
        \omega>\BB(m),
        b<K(\omega)\le b+d
      \right\}
    \right)
    \ge
    2^{-m-K(b\mid m)-\cref{c:BB-window}}.
  \end{equation}
\end{lemma}

\begin{proof}
  Let $\BB(m)^*$ be a shortest description of $\BB(m)$
  (it does not matter which, since $\m=O(2^{-K})$,
  and so there are only $O(1)$ shortest descriptions).
  Remember that
  \begin{equation}\label{eq:BB}
    m-O(1)\le K(\BB(m))\le m,
  \end{equation}
  where the first inequality follows from $K(\BB(m)+1)>m$ (see~\eqref{eq:B}).
  Consider the objects $\omega_z:=2^{\BB(m)}(2z+1)$, $z\in\N$.
  Applying Lemma~\ref{lem:window} to suitable $z$s
  will provide us with $\omega$s satisfying~\eqref{eq:BB-window}.

  Since $\BB(m)$ and $z$ can be computed from $\omega_z$,
  the complexity of a pair theorem
  (see, e.g., \cite[Theorem~67]{Shen/etal:2017book})
  gives a constant $\cnew$\clabel{c:pair} such that,
  for all $m\ge K(1)$ and all $z\in\N$,
  \begin{equation}\label{eq:K-omega-z}
    \left|
      K(\omega_z)-m-K(z\mid \BB(m)^*)
    \right|
    \le
    \cref{c:pair}.
  \end{equation}
  Set $k:=b-m+\cref{c:pair}$.
  By~\eqref{eq:K-omega-z},
  every $z$ satisfying the condition $k<K(z\mid \BB(m)^*)\le k+d_0$
  in Lemma~\ref{lem:window}
  also satisfies $b<K(\omega_z)\le b+d_0+2\cref{c:pair}$;
  besides, $\omega_z\ge2^{\BB(m)}>\BB(m)$.
  Set $d:=d_0+2\cref{c:pair}$.
  Since $k\ge b-m\ge k_0$,
  Lemma~\ref{lem:window} with $y:=\BB(m)^*$ and~\eqref{eq:K-omega-z}
  show that the universal measure of these $\omega_z$ is at least
  \[
    \sum_{z:\,k<K(z\mid \BB(m)^*)\le k+d_0}
    2^{-m-K(z\mid \BB(m)^*)-O(1)}
    \ge
    2^{-m-K(k\mid \BB(m)^*)-O(1)}.
  \]
  It remains to notice that $K(k\mid \BB(m)^*)\le K(b\mid m)+O(1)$,
  since $m$ is determined by the length $K(\BB(m))$ of $\BB(m)^*$
  up to $O(1)$ possibilities (see \eqref{eq:BB}).
\end{proof}

Now we have everything we need to prove Theorem~\ref{thm:opt}.

\begin{proof}[Proof of Theorem~\ref{thm:opt}]
  Without loss of generality,
  $\cref{c:max}\in\N$ and $\cref{c:max}\ge K(1)$
  in Lemma~\ref{lem:max}.
  Since $g$ is computable, increasing, and unbounded,
  there is a computable function $A:\N^2\to\N_0$
  (informally, $A(m,e)\approx g^{-1}(m-e)\approx g^{-1}(m-e-1)$)
  such that, for all $e,k,m\in\N$,
  \begin{equation}\label{eq:A}
    g(k)\le m-e-1
    \Longrightarrow
    k\le A(m,e)
    \Longrightarrow
    g(k)\le m-e.
  \end{equation}
  For example, $A(m,e)$ can be defined as the largest $k\in\N_0$
  such that $\hat g(j)\le m-e-1/2$ for all $j\in\{1,\dots,k\}$,
  where $\hat g(j)$ is a rational approximation to $g(j)$ with accuracy $1/2$ computed from $j$.
  Let $k_0$ and $d$ be as in Lemma~\ref{lem:BB-window},
  and let $m_0\in\N_0$ satisfy the defining inequality
  $g(\max(m_0,a)+\cref{c:opt})\le m_0$ in~\eqref{eq:opt-1}
  (eventually, we will take $m_0:=m_g(a)$, assuming $m_g(a)<\infty$),
  where $\cref{c:opt}$ is a constant, to be chosen later,
  that is large as compared with $\cref{c:max}$, $k_0$, and $d$.
  Set $m:=m_0+\cref{c:max}+1$ and $A^*:=A(m,\cref{c:max})\approx g^{-1}(m_0)$.

  The key reduction of \eqref{eq:opt-2} in Theorem~\ref{thm:opt}
  to \eqref{eq:BB-window} in Lemma~\ref{lem:BB-window}
  will be based on the inequality
  \begin{equation}\label{eq:reduction}
    \m
    \left(
      K>a,
      \beta(g(K))=\infty
    \right)
    \ge
    \m
    \left(
      \left\{
        \omega
        \st
        \omega>\BB(m),
        A^*-d<K(\omega)\le A^*
      \right\}
    \right).
  \end{equation}
  Let us check that the set on the right-hand side is
  a subset of the set on the left-hand side.
  By the right-hand implication of~\eqref{eq:A} with $k:=K(\omega)$ and $e:=\cref{c:max}$,
  every $\omega$ satisfying $K(\omega)\le A^*$ also satisfies $g(K(\omega))\le m-\cref{c:max}$;
  if, in addition, $\omega>\BB(m)$,
  then $\omega>\BB(g(K(\omega))+\cref{c:max})$,
  and so $\beta_{\omega}(g(K(\omega)))=\infty$ by Lemma~\ref{lem:max}.
  By the left-hand implication of~\eqref{eq:A} with $k:=\max(m_0,a)+\cref{c:opt}$ and $e:=\cref{c:max}$,
  we have $\max(m_0,a)+\cref{c:opt}\le A^*$;
  in particular, $K(\omega)>A^*-d$ implies $K(\omega)>a$.
  This proves \eqref{eq:reduction}.

  Lemma~\ref{lem:BB-window} is applicable to the right-hand side of \eqref{eq:reduction}
  since $m>\cref{c:max}\ge K(1)$
  (we assumed $\cref{c:max}\ge K(1)$ at the start of the proof)
  and $A^*-d\ge m+k_0$ (for a large enough $\cref{c:opt}$),
  and it gives
  \[
    \m
    \left(
      K>a,
      \beta(g(K))=\infty
    \right)
    \ge
    2^{-m-K(A^*-d\mid m)-\cref{c:BB-window}}
    \ge
    2^{-m_0-\cref{c:opt}},
  \]
  where the last inequality holds for a large enough $\cref{c:opt}$,
  since $m-m_0$ and $K(A^*-d\mid m)$ are bounded by constants that do not depend on $\cref{c:opt}$.
  It remains to set $m_0:=m_g(a)$.
\end{proof}

\section{Statements with the conditional universal measure}

In this section we state Corollary~\ref{cor:set-g} in terms of the conditional universal measure
given a window of complexities.

\begin{theorem}\label{thm:set-g-window}
  Let $g$ be as in Corollary~\ref{cor:set-g}.
  There are positive constants $c$ and $d$ such that,
  from some $a$ on,
  \begin{multline}\label{eq:set-g-window}
    2^{-g(a)-c}
    \le
    \m
    \left(
      \beta(g(K))=\infty
      \mid
      a<K\le a+d
    \right)\\
    \le
    \m
    \left(
      \beta(g(K))>c
      \mid
      a<K\le a+d
    \right)
    \le
    2^{-g(a)+c\log a}.
  \end{multline}
\end{theorem}

\noindent
As usual in probability theory, $\m(E\mid F):=\m(E\cap F)/\m(F)$.
(There is some ambiguity in the notation $\m(\cdot\mid\cdot)$
since it can also mean the universal semimeasure
of the first argument when the second argument is given
as an input to the relevant universal probabilistic machine,
as in \cite[Sect.~4.7.1]{Shen/etal:2017book}.
In this paper we will only use this notation
in the usual sense of probability theory.)

In the following proof of Theorem~\ref{thm:set-g-window}
we will in fact establish a more general statement
containing both Theorem~\ref{thm:set-g-window} and a version of Corollary~\ref{cor:set-g}
(slightly weaker than \eqref{eq:set-g-bis})
as special cases.
Namely, we will replace \eqref{eq:set-g-window}
by
\begin{multline}\label{eq:window}
  \forall a'\in\{a+d,a+d+1,\dots\}\cup\{\infty\}:\\
  2^{-g(a)-c}
  \le
  \m
  \left(
    \beta(g(K))=\infty
    \mid
    a<K\le a'
  \right)\\
  \le
  \m
  \left(
    \beta(g(K))>c
    \mid
    a<K\le a'
  \right)
  \le
  2^{-g(a)+c\log a},
\end{multline}
where $a<K\le\infty$ means $K>a$.

\begin{proof}[Proof of Theorem~\ref{thm:set-g-window}]
  We will use the following implications of~\eqref{eq:regular}, valid for large $x$:
  $g(x+t)\le g(x)+t$ for all $t\ge0$,
  $g(x)\le(1-\epsilon)x+O(1)$,
  and $1/g'(x)\le x$.
  Let $d_0$ and $k_0$ be as in Lemma~\ref{lem:window}
  and $d\ge d_0$ be as in Lemma~\ref{lem:BB-window}.

  First we prove the upper bound in~\eqref{eq:window} for all $a'\ge a+d_0$.
  By Theorem~\ref{thm:main} as elaborated in Corollary~\ref{cor:set-g},
  \begin{multline*}
    \m
    \left(
      a<K\le a',
      \beta(g(K))>\cref{c:main}
    \right)
    \le
    \m
    \left(
      K>a,
      \beta(g(K))>\cref{c:main}
    \right)\\
    =
    O\left(2^{-g(a)}/g'(a)\right)
    =
    O(a2^{-g(a)}),
  \end{multline*}
  while, by Lemma~\ref{lem:window},
  \[
    \m(a<K\le a')
    \ge
    \m(a<K\le a+d_0)
    \ge
    2^{-K(a)-\cref{c:window}}
    \ge
    a^{-2}
  \]
  for large $a$.
  Dividing the former by the latter gives the upper bound in~\eqref{eq:window}
  for large $a$,
  provided the constant $c$ is sufficiently large.

  For the lower bound,
  let $m$ and $A^*$ be as in the proof of Theorem~\ref{thm:opt}
  with $m_0:=m_g(a)$.
  For $b\in\N$,
  let $W_b$ be the set of all $\omega>\BB(m)$ satisfying $b<K(\omega)\le b+d$
  (as in \eqref{eq:BB-window}).
  By the proof of Corollary~\ref{cor:set-g},
  $m=m_g(a)+\cref{c:max}+1\le g(a)+O(1)$ for large $a$;
  in particular, $a\ge m+k_0$ for large $a$,
  and then, by Lemma~\ref{lem:BB-window},
  \begin{equation}\label{eq:W}
    \m(W_b)
    \ge
    2^{-m-K(b\mid m)-\cref{c:BB-window}}
  \end{equation}
  holds for all $b\ge a$.
  We will show that, in each of the following two cases,
  \begin{equation}\label{eq:window-m}
    \m
    \left(
      \beta(g(K))=\infty
      \mid
      a<K\le a'
    \right)
    \ge
    2^{-m-O(1)};
  \end{equation}
  since $m\le g(a)+O(1)$,
  this will prove the lower bound in~\eqref{eq:window}.

  If $a'\ge A^*$ (in particular, if $a'=\infty$),
  all elements $\omega$ of $W_{A^*-d}$ satisfy $a<K(\omega)\le A^*\le a'$
  and $\beta_{\omega}(g(K(\omega)))=\infty$, and
  the proof of Theorem~\ref{thm:opt} (see \eqref{eq:reduction})
  shows that $\m(W_{A^*-d})\ge2^{-m-O(1)}$;
  it remains to use $\m(a<K\le a')\le1$.

  If $a+d\le a'<A^*$,
  set $I:=\lfloor(a'-a)/d\rfloor\ge1$ and $b_i:=a+(i-1)d$ for $i\in\{1,\dots,I\}$.
  The sets $W_{b_i}$ are disjoint,
  and all their elements $\omega$ satisfy $a<K(\omega)\le a'$
  and (since $a'<A^*$) $\beta_{\omega}(g(K(\omega)))=\infty$.
  Therefore, by~\eqref{eq:W} and $K(b_i\mid m)\le K(b_i)+O(1)$,
  \[
    \m
    \left(
      a<K\le a',
      \beta(g(K))=\infty
    \right)
    \ge
    \sum_{i=1}^{I}
    2^{-m-K(b_i)-O(1)}.
  \]
  On the other hand,
  for each $k\in\{a+1,\dots,a'\}$
  there is $i\in\{1,\dots,I\}$ such that $0<k-b_i\le2d$
  (we use $2d$ rather than $d$ to account for the last block
  into which the $b_i$ split the interval $(a,a']$)
  and each $i$ is obtained in this way for at most $2d$ values of $k$.
  Since $K(k)\ge K(b_i)-O(1)$ when $0<k-b_i\le2d$,
  the last statement of Lemma~\ref{lem:window} gives
  \[
    \m(a<K\le a')
    =
    \sum_{k=a+1}^{a'}
    \m(K=k)
    \le
    \sum_{k=a+1}^{a'}
    2^{-K(k)+\cref{c:window}}
    \le
    2^{O(1)}
    \sum_{i=1}^{I}
    2^{-K(b_i)}.
  \]
  Dividing the former display by the latter gives~\eqref{eq:window-m}.
\end{proof}

\begin{remark}
  The bounds \eqref{eq:set-g} in Corollary~\ref{cor:set-g}
  cover the objects of complexity greater than $a$,
  and objects with enormous complexities make a significant contribution
  to their universal measure $\m(K>a)$:
  by Lemmas~3 and~4 in \cite{Vovk:arXiv2608-local},
  $\m(K>G(a)\mid K>a)\ge c_G$ for every computable function $G$
  and some positive constant $c_G$.
  One might therefore suspect that Corollary~\ref{cor:set-g}
  crucially depends on those very complex objects.
  Theorem~\ref{thm:set-g-window} compares like with like
  and shows that this is not the case:
  the rate $2^{-g(a)+O(\log a)}$ holds within every window $(a,a+d]$ of constant width.
  Therefore, for $\log a=o(g(a))$, probabilistically nonrandom objects are rare at every level of complexity,
  not just on average.
\end{remark}

\begin{remark}
  For $g(k):=\theta k$ with a computable $\theta\in(0,1)$,
  Theorem~\ref{thm:set-g-window} can be sharpened:
  there is a positive constant $c$ such that, from some $a$ on,
  \begin{multline}\label{eq:linear}
    2^{-\theta a+K(a)-c}
    \le
    \m
    \left(
      \beta(\theta K)=\infty
      \mid
      a<K\le a+d
    \right)\\
    \le
    \m
    \left(
      \beta(\theta K)>c
      \mid
      a<K\le a+d
    \right)
    \le
    2^{-\theta a+K(a)+c}.
  \end{multline}
  For the upper bound,
  divide the upper bound $2^{-\theta a+c}$ in~\eqref{eq:set-1}
  by $\m(a<K\le a+d)\ge\m(a<K\le a+d_0)\ge2^{-K(a)-\cref{c:window}}$
  (Lemma~\ref{lem:window}).

  For the lower bound,
  let $m$, $A^*$, and $W_b$ be as in the proof of Theorem~\ref{thm:set-g-window}.
  All elements $\omega$ of $W_a$ satisfy $a<K(\omega)\le a+d\le A^*$
  and, therefore, $\beta_{\omega}(\theta K(\omega))=\infty$.
  For large $a$, we have $m=\lceil\theta(a+\cref{c:opt})\rceil+\cref{c:max}+1$,
  where $\cref{c:opt}$ is the constant in Theorem~\ref{thm:opt}.
  Therefore, $K(a\mid m)=O(1)$,
  and~\eqref{eq:W} gives
  \[
    \m(W_a)\ge2^{-m-O(1)}\ge2^{-\theta a-O(1)}.
  \]
  On the other hand,
  the last statement of Lemma~\ref{lem:window} gives
  \[
    \m(a<K\le a+d)
    \le
    \sum_{k=a+1}^{a+d}
    2^{-K(k)+\cref{c:window}}
    \le
    2^{-K(a)+O(1)}.
  \]
  Dividing the former display by the latter gives the lower bound in~\eqref{eq:linear}.

  Since $K(a)\ge\log a-O(1)$ for most $a$,
  the conditional probabilities in~\eqref{eq:linear}
  are typically at least of order $a2^{-\theta a}$.
  Therefore, the term $c\log a$ in the upper bound in~\eqref{eq:set-g-window}
  cannot be replaced by $o(\log a)$,
  even though $\m(K>a,\beta(\theta K)=\infty)=2^{-\theta a+O(1)}$ by~\eqref{eq:set-1}.
  On the other hand, for simple $a$, such as $a=2^j$ for $j\in\N$,
  the conditional probability in~\eqref{eq:linear}
  is $2^{-\theta a+O(\log\log a)}$.
\end{remark}

\section{Probability measures as models}

In Sect.~\ref{sec:set} we used finite sets in $\N$ as models.
Another common type of models is computable probability measures on $\N$.
In the standard theory of algorithmic statistics
the difference between the two kinds of models is considered negligible
for large $\alpha$ but becomes more pronounced as $\alpha$ shrinks;
see, e.g., \cite[Theorem~1]{Shen:1983} and \cite[Proposition~2]{Vereshchagin/Shen:2017}.
In this section we will see that,
in the range of the bounds on model complexity used in~\eqref{eq:set-1}--\eqref{eq:set-4},
replacing finite sets by probability measures changes nothing essential:
the same rates hold.
Below that range, i.e., for $\alpha<\log K(\omega)$,
the difference between the two kinds of models becomes substantial
(Sect.~\ref{sec:below}).

It is customary in the standard theory to consider only rational-valued probability measures
with finite support; see, e.g., \cite{Shen:1983} and \cite[Definition~1]{Vereshchagin/Shen:2017}.
Such probability measures $Q$ on $\N$ are called \emph{elementary};
they are finite objects, and so $K(Q)$ and $K(\omega\mid Q)$ are defined.
The \emph{randomness deficiency} of $\omega\in\N$ w.r.t.\ an elementary $Q$ is
\[
  d(\omega\mid Q)
  :=
  -\log Q(\{\omega\})-K(\omega\mid Q),
\]
with $-\log0:=\infty$,
and the best fit function with elementary probability measures as models is
\[
  \beta^{\mathrm{q}}_{\omega}(\alpha)
  :=
  \min_{Q:K(Q)\le\alpha}
  d(\omega\mid Q),
\]
$Q$ ranging over the elementary probability measures.

More general definitions are also common in the literature;
see, e.g., \cite{Gacs/etal:2001} and \cite[Appendix~II]{Vereshchagin/Vitanyi:2004}.
Let us fix a G\"odel numbering of all algorithms
that accept pairs in $\N\times\N$ as input and output rational numbers.
We say that $\pi\in\N$ is a \emph{number} of a probability measure $P$ on $\N$
if it is a number of an algorithm that,
given $\omega\in\N$ and $n\in\N$,
outputs a rational number within $2^{-n}$ of $P(\{\omega\})$.
For a number $\pi$ of $P$ we set
\[
  d(\omega\mid\pi)
  :=
  -\log P(\{\omega\})-K(\omega\mid\pi),
\]
and the best fit function with computable probability measures as models is
\[
  \beta^{\mathrm{Q}}_{\omega}(\alpha)
  :=
  \min_{\pi:K(\pi)\le\alpha}
  d(\omega\mid\pi),
\]
$\pi$ ranging over the numbers of probability measures on $\N$.
In the terminology of \cite{Vovk:arXiv2608-local},
$\omega$ is \emph{$(\alpha,\beta)$-quasi-stochastic} if $\beta^{\mathrm{q}}_{\omega}(\alpha)\le\beta$
and \emph{$(\alpha,\beta)$-Quasi-stochastic} (with a capital Q)
if $\beta^{\mathrm{Q}}_{\omega}(\alpha)\le\beta$.
Like $\beta_{\omega}$, both functions are decreasing
and constant on each interval $[n,n+1)$, $n\in\N_0$.

Regarding a finite set as the uniform probability measure on it,
and computing a number of an elementary probability measure from the measure,
we obtain, for a constant $\cnew$\clabel{c:pm-in} and all $\omega$ and $\alpha$,
\begin{equation}\label{eq:pm-in}
  \beta^{\mathrm{q}}_{\omega}(\alpha+\cref{c:pm-in})
  \le
  \beta_{\omega}(\alpha)+\cref{c:pm-in}
  \quad\text{and}\quad
  \beta^{\mathrm{Q}}_{\omega}(\alpha+\cref{c:pm-in})
  \le
  \beta^{\mathrm{q}}_{\omega}(\alpha)+\cref{c:pm-in}.
\end{equation}
The \emph{support} $\supp Q$ of a probability measure $Q$
is defined as $\{\omega\st Q(\{\omega\})>0\}$.
Since $\max\supp Q$ is computable from an elementary probability measure $Q$,
there is a constant $\cnew$\clabel{c:max-q} such that
$K(\max\operatorname{supp}Q)\le K(Q)+\cref{c:max-q}$
for all elementary probability measures $Q$;
therefore, the proof of Lemma~\ref{lem:max} shows that
\begin{equation}\label{eq:max-q}
  \beta^{\mathrm{q}}_{\omega}(\alpha)=\infty
  \quad\text{whenever}\quad
  \omega>\BB(\alpha+\cref{c:max-q}).
\end{equation}
In contrast, $\beta^{\mathrm{Q}}_{\omega}(\alpha)<\infty$ for all $\omega$
as soon as $\alpha$ exceeds the complexity of a number of the probability measure
$P(\{\omega\}):=2^{-\omega}$, $\omega\in\N$.
For $\beta^{\mathrm{Q}}$,
we will also need the inverse
\[
  \mu(t)
  :=
  \min\{K(n)\st n>t\},
  \quad
  t\ge0,
\]
of the busy-beaver function $\BB$;
we have $\mu(t)>j$ if and only if $\BB(j)\le t$.
The function $\mu$ is increasing and unbounded
but grows extremely slowly:
$\mu$ is eventually below any unbounded increasing computable function;
e.g., $\mu(2^{2^j})\le K(j)+O(1)$.

Now we can state the analogues of Theorems~\ref{thm:main} and~\ref{thm:opt}
for probability measures as models,
to be proved in Sect.~\ref{sec:pm-proofs}.

\begin{theorem}\label{thm:pm-main}
  Theorem~\ref{thm:main} remains true
  if $\beta_{\omega}$ is replaced by $\beta^{\mathrm{q}}_{\omega}$ or by $\beta^{\mathrm{Q}}_{\omega}$
  (with different constants).
\end{theorem}

\begin{theorem}\label{thm:pm-opt}
  Theorem~\ref{thm:opt} remains true
  if $\beta_{\omega}$ is replaced by $\beta^{\mathrm{q}}_{\omega}$
  (with a different constant in place of $\cref{c:opt}$).
  Moreover, for each $g$ as in Theorem~\ref{thm:opt}
  there is a positive integer constant $\cnew$\clabel{c:pm-opt} such that,
  for all $\beta\ge0$ and all $a\in\N_0$,
  \begin{equation}\label{eq:pm-opt}
    \m
    \left(
      K>a,
      \beta^{\mathrm{Q}}(g(K))
      >
      \beta
    \right)
    \ge
    2^{-m^{\beta}_g(a)-K(\mu(\beta))-\cref{c:pm-opt}},
  \end{equation}
  where $m^{\beta}_g(a)$ is defined
  by the following modification of~\eqref{eq:opt-1}:
  \begin{equation}\label{eq:pm-opt-1}
    m^{\beta}_g(a)
    :=
    \min
    \left\{
      m\in\N_0
      \st
      g
      \left(
        \max(m,a)+\cref{c:pm-opt}
      \right)
      +
      \mu(\beta)
      \le
      m
    \right\}.
  \end{equation}
\end{theorem}

\noindent
For $\beta^{\mathrm{Q}}$,
``$=\infty$'' has to be replaced by ``$>\beta$'' in Theorem~\ref{thm:pm-opt}.
Indeed, as noted above,
$\beta^{\mathrm{Q}}_{\omega}(\alpha)<\infty$ for all $\omega$
as soon as $\alpha$ exceeds a constant;
therefore, e.g., for $g(k):=\theta\log k$
the left-hand side of~\eqref{eq:opt-2} with $\beta^{\mathrm{Q}}$ in place of $\beta$
is $0$ for all sufficiently large $a$,
whereas its right-hand side is positive.
The witnesses in~\eqref{eq:pm-opt}
can have huge randomness deficiencies,
and the dependence of the right-hand side of~\eqref{eq:pm-opt} on $\beta$ is very weak
(only through $\mu(\beta)$).

We also have the following analogue of Corollary~\ref{cor:set-g}.

\begin{corollary}\label{cor:pm-g}
  Corollary~\ref{cor:set-g} remains true
  if $\beta$ is replaced by $\beta^{\mathrm{q}}$.
  Let $g$ be as in Corollary~\ref{cor:set-g}.
  There is a positive constant $c$ such that,
  for all $\beta\ge c$ and all $a\ge c(\mu(\beta)+1)$,
  \begin{equation}\label{eq:g-Q}
    2^{-g(a)-\mu(\beta)-K(\mu(\beta))-c}
    \le
    \m
    \left(
      K>a,
      \beta^{\mathrm{Q}}(g(K))>\beta
    \right)
    \le
    2^{-g(a)+c} / g'(a)
  \end{equation}
  (for the upper bound, the condition $a\ge c$ suffices).
\end{corollary}

\noindent
As in Theorem~\ref{thm:pm-opt},
``$>\beta$'' cannot be replaced by ``$=\infty$'' in~\eqref{eq:g-Q}:
the second inequality in~\eqref{eq:regular} implies $1/g'(x)=O(x)$,
and so $g(x)\to\infty$ as $x\to\infty$,
which makes $\m(K>a,\beta^{\mathrm{Q}}(g(K))=\infty)$ vanish for large $a$.

\begin{proof}
  The upper bounds follow from Theorem~\ref{thm:pm-main}
  in the same way as the upper bound in Corollary~\ref{cor:set-g}
  follows from Theorem~\ref{thm:main}
  (in~\eqref{eq:g-Q}, we also use $\beta\ge c$),
  and the lower bound for $\beta^{\mathrm{q}}$ follows from Theorem~\ref{thm:pm-opt}
  in the same way as the lower bound in Corollary~\ref{cor:set-g}
  follows from Theorem~\ref{thm:opt}.
  For the lower bound in~\eqref{eq:g-Q},
  let $m:=\lceil g(a+\cref{c:pm-opt})+\mu(\beta)\rceil$,
  $\cref{c:pm-opt}$ being the constant in Theorem~\ref{thm:pm-opt}.
  By~\eqref{eq:regular},
  $m\le(1-\epsilon)a+\mu(\beta)+O(1)$ for large $a$
  (in a sense that does not depend on $\beta$),
  and so $m\le a$ when $a\ge c(\mu(\beta)+1)$ and $c$ is large enough.
  Then $m$ satisfies the defining inequality \eqref{eq:pm-opt-1} for $m^{\beta}_g(a)$,
  and~\eqref{eq:pm-opt} gives the lower bound
  \[
    2^{-m-K(\mu(\beta))-\cref{c:pm-opt}}
    \ge
    2^{-g(a)-\mu(\beta)-K(\mu(\beta))-O(1)}.
    \qedhere
  \]
\end{proof}

Since \eqref{eq:set-1}--\eqref{eq:set-4} are special cases of Corollary~\ref{cor:set-g},
Corollary~\ref{cor:pm-g} implies the following analogues.

\begin{corollary}\label{cor:pm-q}
  Inequalities~\eqref{eq:set-1}--\eqref{eq:set-4} remain true
  if $\beta(\cdot)$ is replaced by $\beta^{\mathrm{q}}(\cdot)$.
\end{corollary}

\begin{corollary}\label{cor:pm-Q}
  The upper bounds in~\eqref{eq:set-1}--\eqref{eq:set-4} remain true
  if $\beta(\cdot)$ is replaced by $\beta^{\mathrm{Q}}(\cdot)$.
  Let $\beta\ge0$.
  The lower bounds in~\eqref{eq:set-1}--\eqref{eq:set-4} remain true
  if ``$\beta(\cdot)=\infty$'' is replaced by ``$\beta^{\mathrm{Q}}(\cdot)>\beta$''
  (the constants now depending on the threshold $\beta$).
\end{corollary}

Therefore, the rates in~\eqref{eq:set-1}--\eqref{eq:set-4}
are optimal for probability measures as models
in the same sense as for finite sets.

\begin{proof}[Proof of Corollaries~\ref{cor:pm-q} and~\ref{cor:pm-Q}]
  Since \eqref{eq:set-1}--\eqref{eq:set-4} are obtained
  by specializing $g$ in Corollary~\ref{cor:set-g},
  the same specializations of Corollary~\ref{cor:pm-g}
  give Corollaries~\ref{cor:pm-q} and~\ref{cor:pm-Q}.
  For the second part of Corollary~\ref{cor:pm-Q},
  small thresholds $\beta$ are dealt with by the universal measure in~\eqref{eq:g-Q}
  being decreasing in $\beta$,
  and small $a$ are absorbed into the constant $c$
  as the universal measure is positive.
\end{proof}

In the same way,
we obtain the following analogue of Theorem~\ref{thm:set-g-window},
which will be proved in the next section.

\begin{theorem}\label{thm:pm-g-window}
  Theorem~\ref{thm:set-g-window} remains true
  if $\beta(\cdot)$ is replaced by $\beta^{\mathrm{q}}(\cdot)$.
  Let $g$ be as in Corollary~\ref{cor:set-g}.
  There are positive constants $c$ and $d$ such that,
  for all $\beta\ge c$ and all $a\ge c(\mu(\beta)+1)$,
  \begin{equation}\label{eq:pm-g-window}
    2^{-g(a)-\mu(\beta)-K(\mu(\beta))-c}
    \le
    \m
    \left(
      \beta^{\mathrm{Q}}(g(K))>\beta
      \mid
      a<K\le a+d
    \right)
    \le
    2^{-g(a)+c\log a}.
  \end{equation}
\end{theorem}

\begin{remark}
  The main message of \cite{Vovk:arXiv2608-local} is that
  there is a chasm between stochasticity and quasi-stochasticity
  in the universal measure of objects $\omega$ that fail to satisfy these properties.
  The key difference, already alluded to in Sect.~\ref{sec:intro},
  that explains the disappearance of the chasm in our current setting
  is that in \cite{Vovk:arXiv2608-local} we are interested in $\alpha$ and $\beta$
  bounded by functions of the threshold on $K(\omega)$
  ($k$ in \cite{Vovk:arXiv2608-local} and $a$ in this paper)
  such as $2^{2^k}$,
  while in this paper we are interested in $\alpha$ as a function of $K(\omega)$.
  Under the universal measure $\m$,
  the difference between the threshold and $K(\omega)$ can be huge with significant probability.
  By Lemmas~3 and~4 in \cite{Vovk:arXiv2608-local},
  $\m(K>G(k)\mid K>k)\ge c_G>0$ for every computable $G$.
  For example, a constant fraction of the objects with $K(\omega)>k$ have $K(\omega)>2^{2^{2^k}}$,
  and for them the argument $\theta\log K(\omega)$ in~\eqref{eq:set-4}
  exceeds $2^{2^k}$.
\end{remark}

\section{Proofs of Theorems~\ref{thm:pm-main}, \ref{thm:pm-opt}, and~\ref{thm:pm-g-window}}
\label{sec:pm-proofs}

\begin{proof}[Proof of Theorem~\ref{thm:pm-main}]
  The proof of Theorem~\ref{thm:main} in Sect.~\ref{sec:set-proofs}
  works verbatim for $\beta^{\mathrm{q}}$ and $\beta^{\mathrm{Q}}$
  if each block $\Omega$ is replaced by the uniform probability measure on $\Omega$
  or by a number of that measure computed from $\Omega$.
  (Alternatively, we could deduce Theorem~\ref{thm:pm-main} from Theorem~\ref{thm:main}
  using~\eqref{eq:pm-in}.)
\end{proof}

The key to the lower bounds is the following counterpart,
for computable probability measures,
of the inequality $K(\max\Omega)\le K(\Omega)+\cref{c:max}$
used in the proof of Lemma~\ref{lem:max}.
It is a version of \cite[Lemma~21]{Vovk:arXiv2608-local}
that bounds the randomness deficiency rather than the tail.

\begin{lemma}\label{lem:tail}
  There is a constant $\cnew$\clabel{c:tail} such that,
  for every number $\pi$ of a probability measure $P$ on $\N$ and every $s\in\N$,
  there exists $N\in\N$ satisfying $K(N)\le K(\pi)+K(s)+\cref{c:tail}$ and
  \[
    \forall\omega>N:
    \quad
    d(\omega\mid\pi)
    \ge
    s-K(s)-\cref{c:tail}.
  \]
\end{lemma}

\begin{proof}
  Given $\pi$ and $s$,
  we can compute $N\in\N$ such that $P(\{n\st n>N\})<2^{-s}$:
  compute the values of $P$ on singletons with increasing accuracy
  until we obtain a lower bound on $P(\{1,\dots,N\})$
  exceeding $1-2^{-s}$ for some $N$;
  this will happen since $P$ is a probability measure.
  Therefore, $K(N)\le K(\pi)+K(s)+O(1)$.
  The function equal to $2^sP(\{\omega\})$ for $\omega>N$ and to $0$ for $\omega\le N$
  is a semimeasure (its total mass is less than 1)
  computable given $\pi$ and $s$,
  and so
  $K(\omega\mid\pi,s)\le-\log P(\{\omega\})-s+O(1)$ for all $\omega>N$.
  Therefore,
  \begin{multline*}
    d(\omega\mid\pi)
    =
    -\log P(\{\omega\})-K(\omega\mid\pi)\\
    \ge
    -\log P(\{\omega\})-K(\omega\mid \pi,s)-K(s\mid\pi)-O(1)
    \ge
    s-K(s)-O(1).
    \qedhere
  \end{multline*}
\end{proof}

\begin{lemma}\label{lem:max-Q}
  There is a constant $\cnew$\clabel{c:max-Q} such that,
  for all $\omega\in\N$, $\alpha\ge0$, and $\beta\ge0$,
  we have $\beta^{\mathrm{Q}}_{\omega}(\alpha)>\beta$
  whenever $\omega>\BB(\alpha+\mu(\beta)+\cref{c:max-Q})$.
\end{lemma}

\begin{proof}
  Let $\omega$, $\alpha$, and $\beta$ be as in the statement of the lemma.
  Let $n>\beta$ satisfy $K(n)=\mu(\beta)$
  and set $s:=2n+\cnew$ with $\cold$ so large that
  $s-K(s)-\cref{c:tail}>n>\beta$;
  we have $K(s)\le\mu(\beta)+O(1)$.
  Let $\pi$ be a number of a probability measure with $K(\pi)\le\alpha$,
  and let $N$ be as in Lemma~\ref{lem:tail}.
  Then $K(N)\le\alpha+\mu(\beta)+\cref{c:max-Q}$ for a large enough $\cref{c:max-Q}$,
  which implies $N\le \BB(\alpha+\mu(\beta)+\cref{c:max-Q})<\omega$,
  and so Lemma~\ref{lem:tail} gives $d(\omega\mid\pi)\ge s-K(s)-\cref{c:tail}>\beta$.
\end{proof}

\begin{proof}[Proof of Theorem~\ref{thm:pm-opt}]
  The proof of Theorem~\ref{thm:opt} in Sect.~\ref{sec:set-proofs}
  uses the definition of $\beta_{\omega}$ only through Lemma~\ref{lem:max}.
  By~\eqref{eq:max-q}, Lemma~\ref{lem:max} holds for $\beta^{\mathrm{q}}$
  with $\cref{c:max-q}$ in place of $\cref{c:max}$,
  and so Theorem~\ref{thm:opt} holds for $\beta^{\mathrm{q}}$.
  For $\beta^{\mathrm{Q}}$,
  replace $\cref{c:max}$ and Lemma~\ref{lem:max}
  by $\cref{c:max-Q}$ and Lemma~\ref{lem:max-Q} in that proof
  (without loss of generality, $\cref{c:max-Q}\in\N$ and $\cref{c:max-Q}\ge K(1)$),
  let $m_0$ satisfy the defining inequality
  $g(\max(m_0,a)+\cref{c:pm-opt})+\mu(\beta)\le m_0$
  (cf.\ \eqref{eq:pm-opt-1})
  for $m^{\beta}_g(a)$,
  and set $A^*:=A(m,e)$, where $e:=\mu(\beta)+\cref{c:max-Q}$.
  By the left-hand implication of~\eqref{eq:A}
  with $k:=\max(m_0,a)+\cref{c:pm-opt}$
  (notice that $m-e-1=m_0-\mu(\beta)$),
  we have $A^*\ge\max(m_0,a)+\cref{c:pm-opt}$.
  By the right-hand implication of~\eqref{eq:A} with $k:=K(\omega)$,
  every $\omega$ satisfying $K(\omega)\le A^*$ also satisfies $g(K(\omega))\le m-e$;
  if, in addition, $\omega>\BB(m)$,
  then $\omega>\BB(g(K(\omega))+e)$,
  and so $\beta^{\mathrm{Q}}_{\omega}(g(K(\omega)))>\beta$ by Lemma~\ref{lem:max-Q}.
  The only other change is that $A^*-d$ is now computable from $m$ and $e$
  rather than from $m$ alone,
  and so $K(A^*-d\mid m)\le K(e)+O(1)\le K(\mu(\beta))+O(1)$;
  this leads to the term $K(\mu(\beta))$ in~\eqref{eq:pm-opt}.
  None of the constants depend on $\beta$.
\end{proof}

In the following proof of Theorem~\ref{thm:pm-g-window},
we will make the result stronger
as in the proof of Theorem~\ref{thm:set-g-window}:
see \eqref{eq:window}.
Namely, we will prove
\begin{multline}\label{eq:window-Q}
  \forall a'\in\{a+d,a+d+1,\dots\}\cup\{\infty\}:\\
  2^{-g(a)-\mu(\beta)-K(\mu(\beta))-c}
  \le
  \m
  \left(
    \beta^{\mathrm{Q}}(g(K))>\beta
    \mid
    a<K\le a'
  \right)
  \le
  2^{-g(a)+c\log a}
\end{multline}
in place of~\eqref{eq:pm-g-window}.
The upper bound in~\eqref{eq:window-Q} will only need $a\ge c$
as the condition on $a$.

\begin{proof}[Proof of Theorem~\ref{thm:pm-g-window}]
  The proof of the upper bound in~\eqref{eq:window}
  works for $\beta^{\mathrm{q}}$ and $\beta^{\mathrm{Q}}$
  with Theorem~\ref{thm:pm-main} in place of Theorem~\ref{thm:main}
  (for $\beta^{\mathrm{Q}}$, we also use $\beta\ge c$).
  The proof of the lower bound works
  with the changes described in the proof of Theorem~\ref{thm:pm-opt}.
  For $\beta^{\mathrm{Q}}$,
  this gives the lower bound $2^{-m-K(\mu(\beta))-O(1)}$
  in the case $a'\ge A^*$
  and $2^{-m-O(1)}$ in the case $a+d\le a'<A^*$,
  where now $m=m^{\beta}_g(a)+\cref{c:max-Q}+1\le g(a)+\mu(\beta)+O(1)$
  when $a\ge c(\mu(\beta)+1)$
  (see the proof of Corollary~\ref{cor:pm-g})
  and the $O(1)$ do not depend on $\beta$.
\end{proof}

\section{Below the $\log K$ threshold}
\label{sec:below}

For finite sets as models,
the second universal measure in~\eqref{eq:set-4} conditioned on $K>a$ no longer tends to 0 when $\theta<1$:
a constant fraction (as measured by $\m$) of the objects of complexity greater than $a$
is then probabilistically nonrandom,
even with a threshold growing polynomially in $K(\omega)$.
The reason is that the complexity of stochastic objects is simple
(G\'acs et al.\ \cite[Lemmas~III.8 and~III.10]{Gacs/etal:2001}),
as quantified in the following lemma.

\begin{lemma}\label{lem:KK}
  There is a constant $\cnew$\clabel{c:KK} such that,
  for all $\omega\in\N$ and $\alpha\ge0$ satisfying
  $K(K(\omega))>\alpha+2\log(\alpha+2)+\cref{c:KK}$,
  \begin{equation}\label{eq:KK}
    \beta_{\omega}(\alpha)
    \ge
    2^{(K(K(\omega))-\alpha)/2-\cref{c:KK}}-\cref{c:KK}.
  \end{equation}
\end{lemma}

\noindent
Qualitatively, it is easy to see why $K(K(\omega))\approx0$ for stochastic $\omega$:
if $\Omega$ witnesses $\omega$'s stochasticity,
$K(\omega)\approx K(\omega\mid\Omega)\approx\log\left|\Omega\right|$,
and $\log\left|\Omega\right|$ is simple as $\Omega$ is simple.

\begin{proof}[Proof of Lemma~\ref{lem:KK}]
  Let $\Omega$ be a model for $\omega$ with $K(\Omega)\le\alpha$,
  let $\ell:=\lceil\log\left|\Omega\right|\rceil$,
  and let $D:=\ell-K(\omega)$
  ($D$ is known as the unconditional randomness deficiency of $\omega$ in $\Omega$).
  Since $K(\omega)=\ell-D$,
  \begin{equation}\label{eq:KK-star}
    K(K(\omega))
    \le
    K(\ell)+2\log(\left|D\right|+2)+O(1)
    \le
    \alpha+2\log(\left|D\right|+2)+O(1),
  \end{equation}
  the second inequality following from $\ell$ being computable from $\Omega$.
  The concatenation of a shortest description of $\Omega$
  and the index of $\omega$ in $\Omega$ written in exactly $\ell$ bits
  is a description of $\omega$,
  and so $K(\omega)\le K(\Omega)+\ell+O(1)$,
  and so $D\ge-\alpha-O(1)$.
  If $D<0$, then $\left|D\right|\le\alpha+O(1)$,
  which makes~\eqref{eq:KK-star} incompatible
  with the assumption on $K(K(\omega))$ for a large enough $\cref{c:KK}$;
  therefore, $D\ge0$ and, by~\eqref{eq:KK-star},
  $D\ge2^{(K(K(\omega))-\alpha)/2-O(1)}-2$.
  It remains to notice that
  $d(\omega\mid\Omega)\ge D-O(1)$.
\end{proof}

\begin{proposition}\label{prop:below-set}
  Let $\theta\in(0,1)$.
  There is a positive constant $c$ such that, for all $a\in\N_0$,
  \begin{equation}\label{eq:below-set}
    \m
    \left(
      K>a,
      \beta(\theta\log K)
      \ge
      K^{(1-\theta)/4}
    \right)
    \ge
    \tau(a)/c.
  \end{equation}
\end{proposition}

The denominator 2 in~\eqref{eq:KK} can be replaced by any number exceeding~1,
and then the proof will show that the denominator 4 in~\eqref{eq:below-set}
can also be replaced by any number exceeding~1.
By Proposition~\ref{prop:many},
the set in Proposition~\ref{prop:below-set}
contains a constant fraction (as measured by $\m$)
of all objects $\omega$ with $K(\omega)>a$,
and so we can rewrite~\eqref{eq:below-set} as
\[
  \m
  \left(
    \beta(\theta\log K)
    \ge
    K^{(1-\theta)/4}
    \mid K>a
  \right)
  \ge
  1/c
\]
(for a different positive constant $c$).

In the proof of Proposition~\ref{prop:below-set}
we will use the following simple fact
\cite[Lemma~17]{Vovk:arXiv2608-local}.

\begin{lemma}\label{lem:kernel}
  If $\kappa_n$, $n\in\N$, are probability measures on $\N$
  such that $\kappa_n(\{i\})$ is computable from $n$ and $i$,
  then, for some constant $c_{\kappa}>0$,
  \[
    \forall A\subseteq\N:
    \quad
    \m(A)\ge c_{\kappa}\sum_n\m(n)\kappa_n(A).
  \]
\end{lemma}

\begin{proof}
  Indeed, the function $i\mapsto\sum_n\m(n)\kappa_n(\{i\})$
  is a lower semicomputable semimeasure.
\end{proof}

\begin{proof}[Proof of Proposition~\ref{prop:below-set}]
  In this proof we will apply Lemma~\ref{lem:kernel}
  with $\kappa_n$ being the uniform probability measure on $I_n:=\{2^n,\dots,2^{n+1}-1\}$.
  It will allow us to find numerous $\omega$ with large $K(K(\omega))$.

  Set $S:=\{n\in\N\st K(n)\ge\log n\}$
  ($n\in S$ will serve as the sizes of $\omega$ with large $K(K(\omega))$).
  For each $i\in\N$,
  at most $2^{i-K(i)+O(1)}$ elements $n\in I_i$ satisfy $K(n)\le i$
  (see \eqref{eq:Chaitin})
  and all other elements of $I_i$ are in $S$;
  therefore, $\kappa_i(S)\ge1/2$ for all $i\ge i_0$, $i_0$ being a constant.
  Applying Lemma~\ref{lem:kernel} to $S\cap\{n\st n\ge b\}$, we obtain,
  for all $b\in\N$,
  \begin{equation}\label{eq:S}
    \m(S\cap\{n\st n\ge b\})
    \ge
    \frac{c_{\kappa}}{2}\sum_{i\ge\max(\log b,i_0)}\m(i)
    \ge
    \frac{c_{\kappa}}{2}\tau(\max(b,i_0)).
  \end{equation}

  Now let $n\in S$ and $\omega\in I_n$.
  Since $n$ is computable from $\omega$,
  $K(\omega)=K(n)+K(\omega\mid n^*)+O(1)$ \cite[Theorem~67]{Shen/etal:2017book},
  where $n^*$ is a shortest description of $n$;
  since fewer than $2^{n-1}$ objects $\omega$ satisfy $K(\omega\mid n^*)<n-1$,
  at least half of the $\omega\in I_n$ satisfy
  $\left|K(\omega)-n-K(n)\right|\le\cnew$.
  For each such $\omega$ and large $n$,
  $n$ can be computed from $K(\omega)$ and $K(\omega)-n\in[0,2\log n+O(1)]$,
  and so
  \[
    K(K(\omega))\ge K(n)-2\log\log n-O(1)\ge\log n-2\log\log n-O(1),
  \]
  while $\theta\log K(\omega)\le\theta\log n+1$
  (remember that $\theta<1$ and $n$ is supposed to be large).
  For $\alpha:=\theta\log K(\omega)$,
  Lemma~\ref{lem:KK} gives
  \begin{multline*}
    \beta_{\omega}(\theta\log K(\omega))
    \ge
    2^{(\log n-2\log\log n-\theta\log n)/2-O(1)}-O(1)\\
    \ge
    2^{((1-\theta)/3)\log n}
    \ge
    K(\omega)^{(1-\theta)/4}
  \end{multline*}
  for all $n\ge n_0$, $n_0$ being a constant,
  which we also choose so large that $K(\omega)>n$ for $n\ge n_0$.
  Therefore, $\kappa_n(E_a)\ge1/2$ for all $n\in S$ with $n\ge\max(a,n_0)$,
  $E_a$ being the event in the statement of the proposition.
  By Lemma~\ref{lem:kernel} and~\eqref{eq:S},
  \[
    \m(E_a)
    \ge
    \frac{c_{\kappa}}{2}\m(S\cap\{n\st n\ge\max(a,n_0)\})
    \ge
    \frac{c_{\kappa}^2}{4}\tau(\max(a,n_0,i_0)),
  \]
  which is at least $\tau(a)/c$ for a large enough $c$.
\end{proof}

The following proposition shows that,
below the $\log K$ threshold on the model complexity
and for thresholds $\beta$ that are not too small,
$\beta^{\mathrm{Q}}(\cdot)$ and $\beta^{\mathrm{q}}(\cdot)$
behave very differently from $\beta(\cdot)$.

\begin{proposition}\label{prop:below-Q}
  Let $\theta>0$.
  There is a constant $\cnew$\clabel{c:below-Q} such that,
  for all $a\in\N_0$ and $\beta\ge0$,
  \[
    \m
    \left(
      K>a,
      \beta^{\mathrm{Q}}(\theta\log K)
      >
      \beta
    \right)
    \le
    \cref{c:below-Q}
    \left(
      (a+2)^{-\theta}\log^2(a+2)
      +
      2^{-2^{\beta/2-\cref{c:below-Q}}}
    \right),
  \]
  and the same is true for $\beta^{\mathrm{q}}$ in place of $\beta^{\mathrm{Q}}$.
\end{proposition}

\begin{proof}
  Without loss of generality, $a\ge2$.
  Set $\alpha:=\lfloor\theta\log a\rfloor$.
  Since $\beta^{\mathrm{q}}_{\omega}$ is decreasing,
  the set on the left-hand side (with $\beta^{\mathrm{q}}$ in place of $\beta^{\mathrm{Q}}$)
  is contained in the set of objects
  that are not $(\alpha,\beta)$-quasi-stochastic.
  By \cite[Theorem~7]{Vovk:arXiv2608-local},
  the universal measure of the latter set is at most
  $c(2^{-\alpha+K(\alpha)}+2^{-2^{\beta/2-c}})$ for a constant $c$,
  and it remains to notice that $2^{-\alpha+K(\alpha)}=O(a^{-\theta}\log^2a)$.
  The argument for $\beta^{\mathrm{Q}}$ is the same,
  since the upper bound of \cite[Theorem~7]{Vovk:arXiv2608-local}
  also holds for Quasi-stochasticity \cite[end of Sect.~3]{Vovk:arXiv2608-local}.
\end{proof}

The following proposition complements Proposition~\ref{prop:below-Q}
with a lower bound that holds for all computable  $\theta>0$
and whose rate does not depend on the threshold.
It is a special case of Theorems~\ref{thm:opt} and~\ref{thm:pm-opt}.

\begin{proposition}\label{prop:below-lower}
  Let $\theta>0$ be computable.
  There is a positive constant $c$ such that,
  for all $a\in\N_0$,
  \begin{equation}\label{eq:below-lower}
    \m
    \left(
      K>a,
      \beta(\theta\log K)
      =
      \infty
    \right)
    \ge
    (a+2)^{-\theta}/c,
  \end{equation}
  and the same is true for $\beta^{\mathrm{q}}$ in place of $\beta$.
  For each $\beta\ge0$,
  \eqref{eq:below-lower} remains true
  if ``$\beta(\theta\log K)=\infty$'' is replaced by ``$\beta^{\mathrm{Q}}(\theta\log K)>\beta$''
  (the constant $c$ now depending on $\beta$).
\end{proposition}

\begin{proof}
  The function $g(k):=\theta\log k$ is computable, increasing, and unbounded.
  Let $\cref{c:opt}$ be the constant in Theorem~\ref{thm:opt} for this $g$.
  For large $a$,
  $m:=\lceil\theta\log(a+\cref{c:opt})\rceil\le a$,
  and so $m$ satisfies the defining inequality in~\eqref{eq:opt-1};
  therefore, Theorem~\ref{thm:opt} gives
  \[
    \m
    \left(
      K>a,
      \beta(\theta\log K)
      =
      \infty
    \right)
    \ge
    2^{-m-\cref{c:opt}}
    \ge
    2^{-\theta\log(a+\cref{c:opt})-\cref{c:opt}-1},
  \]
  which proves~\eqref{eq:below-lower} for large $a$;
  small $a$ can be absorbed by $c$
  (the left-hand side of~\eqref{eq:below-lower} is positive and decreasing in $a$).
  The statements for $\beta^{\mathrm{q}}$ and $\beta^{\mathrm{Q}}$
  follow from Theorem~\ref{thm:pm-opt} in the same way.
\end{proof}

Combining Propositions~\ref{prop:below-Q} and~\ref{prop:below-lower},
we obtain the rate of decay for elementary probability measures as models
and thresholds $\beta$ that are not too small.

\begin{corollary}\label{cor:below-large}
  Let $\theta>0$ be computable.
  There is a positive constant $c$ such that,
  for all $a\in\N_0$ and all $\beta\ge2\log\log(a+2)+c$,
  \[
    (a+2)^{-\theta}/c
    \le
    \m
    \left(
      K>a,
      \beta^{\mathrm{q}}(\theta\log K)>\beta
    \right)
    \le
    c(a+2)^{-\theta}\log^2(a+2).
  \]
\end{corollary}

\begin{proof}
  The lower bound follows from Proposition~\ref{prop:below-lower}.
  For the upper bound,
  notice that $2^{-2^{\beta/2-\cref{c:below-Q}}}\le(a+2)^{-\theta}$
  as soon as $\beta\ge2\log(\theta\log(a+2))+2\cref{c:below-Q}$
  and apply Proposition~\ref{prop:below-Q}.
\end{proof}

\begin{remark}\label{rem:open}
  By Corollary~\ref{cor:below-large},
  for elementary probability measures as models
  and thresholds $\beta$ growing at least as fast as $2\log\log a$,
  the universal measure of the exceptional set
  is $(a+2)^{-\theta}$ to within a polylogarithmic factor
  for every computable $\theta>0$;
  nothing special happens at $\theta=1$,
  in contrast to finite sets as models (Proposition~\ref{prop:below-set}).
  It is not clear what happens for $\theta\le1$ and probability measures as models
  with a constant threshold $\beta$,
  e.g., whether the second universal measure in \eqref{eq:set-4} conditioned on $K>a$
  still tends to 0 for some $\theta\le1$
  with $\beta^{\mathrm{q}}$ or $\beta^{\mathrm{Q}}$ in place of $\beta$.
\end{remark}

\section{Conclusion}

We have discussed the preponderance of probabilistically random objects,
i.e., objects $\omega$ that are $(\alpha,\beta)$-stochastic
for $\alpha\ll K(\omega)$ and $\beta=O(1)$,
in the structure-free setting.
The approximate picture is as follows.
The most satisfying results were obtained for $\alpha$ above $\log K(\omega)$,
where the conditional universal measure of probabilistically nonrandom objects among objects of complexity $K(\omega)>a$
was shown to be between exponentially small in $a$ for $\alpha=\theta K(\omega)$
and polynomially small in $a$ for $\alpha=\theta\log K(\omega)$
(but $\alpha$ above the $\log K(\omega)$ threshold, i.e., $\theta>1$).
It matters little whether we use finite sets or probability measures
as our models.
The difference between the two classes of models
becomes much more pronounced below the $\log K(\omega)$ threshold for $\alpha$.
This case appears to be more difficult, and important open questions remain;
see, e.g., Remark~\ref{rem:open}.

There are some popular alternatives to best fit functions,
first of all structure functions but also MDL functions,
as defined in \cite[(II.8) and (II.10)]{Vereshchagin/Vitanyi:2004}.
Appendix~\ref{app:structure} briefly discusses structure functions
mainly because of their historical interest,
but a deeper study would be interesting.
Another obvious direction is to replace prefix complexity $K$ used here
by plain Kolmogorov complexity $C$.
Finally, universal e-values used here
(defined by \eqref{eq:deficiency} on the $\log$ scale
for finite sets as models)
can be replaced by universal p-values
defined by $\log\left|\Omega\right|-C(\omega\mid\Omega)$
on the $\log$ scale.
(With probability measures as models,
the relation between universal e-values and universal p-values becomes
more complicated;
see, e.g., \cite[Sect.~5]{Vovk:arXiv2608-local}.)

\subsection*{Acknowledgments}

Claude Opus 5--5.5 and Fable 5.1--5.2 have been used
in exploring proof ideas and developing proofs,
which I reviewed carefully,
and in checking the paper.
I take full responsibility for all claims and statements made here,
including mathematical statements and their proofs.

\appendix
\section{Structure functions}
\label{app:structure}

This appendix is included because of the motivation behind this paper:
Kolmogorov's \cite{Kolmogorov:1974} ``probabilistically random'' objects
are defined in terms of structure functions.
To make it more self-contained,
I will give proofs even of well-known facts
provided the proofs are short.

The \emph{structure function} $h_{\omega}:[0,\infty)\to[0,\infty]$ of an object $\omega\in\N$
is defined as
\[
  h_{\omega}(\alpha)
  :=
  \min_{\Omega:K(\Omega)\le\alpha}
  \log\left|\Omega\right|,
\]
$\Omega$ ranging over the models for $\omega$
(i.e., finite sets $\Omega\ni\omega$)
and $\min\emptyset:=\infty$.
It is decreasing and, $K$ being integer-valued,
constant on each interval $[n,n+1)$, $n\in\N_0$.
We also have $h_{\omega}(K(\omega)+O(1))=0$, as witnessed by $\Omega:=\{\omega\}$.

\begin{lemma}[the sufficiency line]
  There is a constant $\cnew$\clabel{c:line} such that,
  for all $\omega\in\N$ and all $\alpha\ge0$,
  \[
    h_{\omega}(\alpha)\ge K(\omega)-\alpha-\cref{c:line}.
  \]
\end{lemma}

\begin{proof}
  Let $\omega\in\N$ and $\alpha\ge0$.
  Choose a model $\Omega$ for $\omega$ with $K(\Omega)\le\alpha$
  and $\log\left|\Omega\right|=h_{\omega}(\alpha)$
  (there is nothing to prove if $h_{\omega}(\alpha)=\infty$).
  The concatenation of a shortest description of $\Omega$
  with the index of $\omega$ in the sorted list of $\Omega$,
  written in exactly $\lceil\log\left|\Omega\right|\rceil$ bits,
  is a description of $\omega$.
  Therefore,
  \[
    K(\omega)
    \le
    K(\Omega)+\log\left|\Omega\right|+\cref{c:line}
    \le
    \alpha+h_{\omega}(\alpha)+\cref{c:line}.
    \qedhere
  \]
\end{proof}

We call $\alpha\mapsto K(\omega)-\alpha$ the \emph{sufficiency line} of $\omega$ and
\[
  \exc_{\omega}(\alpha)
  :=
  h_{\omega}(\alpha)-K(\omega)+\alpha
  \ \in\
  [-\cref{c:line},\infty]
\]
the \emph{excess} of $\omega$ at $\alpha$.
Objects with a small excess at $\alpha$ are those explained
by a model of complexity $\alpha$
as economically as by their shortest description.
By Lemma~\ref{lem:max},
$\exc_{\omega}(\alpha)=\infty$ whenever $\omega>\BB(\alpha+\cref{c:max})$,
where $\BB$ is the busy-beaver function \eqref{eq:B}.

In the rest of this appendix we will see that Theorem~\ref{thm:set-g-window},
Corollary~\ref{cor:set-g}, and~\eqref{eq:set-1}--\eqref{eq:set-4}
remain true if the best fit function $\beta$ is replaced by the excess $\exc$
and that the resulting statements imply the original ones.
We will need the analogue of Theorem~\ref{thm:main} as a tool.
But we start with the following standard relation
between the excess and the best fit function.

\begin{lemma}\label{lem:exc-beta}
  There is a constant $\cnew$\clabel{c:exc} such that,
  for all $\omega\in\N$ and all $\alpha\ge0$,
  \begin{equation}\label{eq:exc-beta}
    \beta_{\omega}(\alpha)-\cref{c:exc}
    \le
    \exc_{\omega}(\alpha)
    \le
    \beta_{\omega}(\alpha)+\alpha+\cref{c:exc}.
  \end{equation}
  In particular, $\exc_{\omega}(\alpha)=\infty$ if and only if $\beta_{\omega}(\alpha)=\infty$.
\end{lemma}

\begin{proof}
  Let $\omega\in\N$ and $\alpha\ge0$.
  If $\omega$ has no models $\Omega$ with $K(\Omega)\le\alpha$,
  all three terms in~\eqref{eq:exc-beta} are $\infty$;
  so let us assume that it has such models.
  For the first inequality,
  let $\Omega$ be such a model with $\log\left|\Omega\right|=h_{\omega}(\alpha)$.
  Since $K(\omega)\le K(\Omega)+K(\omega\mid\Omega)+O(1)$,
  \[
    \beta_{\omega}(\alpha)
    \le
    \log\left|\Omega\right|-K(\omega\mid\Omega)
    \le
    h_{\omega}(\alpha)-K(\omega)+\alpha+O(1)
    =
    \exc_{\omega}(\alpha)+O(1).
  \]
  For the second inequality,
  letting $\Omega$ be such a model with $d(\omega\mid\Omega)=\beta_{\omega}(\alpha)$,
  we obtain
  \[
    h_{\omega}(\alpha)
    \le
    \log\left|\Omega\right|
    =
    d(\omega\mid\Omega)+K(\omega\mid\Omega)
    \le
    \beta_{\omega}(\alpha)+K(\omega)+O(1),
  \]
  and it remains to subtract the sufficiency line $K(\omega)-\alpha$
  from both sides of the inequality between the extreme terms of this chain.
\end{proof}

By Lemma~\ref{lem:exc-beta},
the events $\beta(g(K))=\infty$ and $\exc(g(K))=\infty$ coincide,
and $\exc(g(K))>g(K)+c$ implies $\beta(g(K))>c-\cref{c:exc}$.
Therefore, Theorem~\ref{thm:set-g-window} immediately implies
the following corollary,
which is rather weak.

\begin{corollary}\label{cor:struct-immediate}
  Let $g$ be as in Corollary~\ref{cor:set-g}.
  There are positive constants $c$ and $d$ such that,
  from some $a$ on,
  \begin{multline*}
    2^{-g(a)-c}
    \le
    \m
    \left(
      \exc(g(K))=\infty
      \mid
      a<K\le a+d
    \right)\\
    \le
    \m
    \left(
      \exc(g(K))>g(K)+c
      \mid
      a<K\le a+d
    \right)
    \le
    2^{-g(a)+c\log a}.
  \end{multline*}
\end{corollary}

Kolmogorov's probabilistically random objects
are those whose excess is small already at a relatively small $\alpha$;
for $\alpha:=g(K)$, we interpret this as $\exc(g(K))\le c$
rather than $\exc(g(K))\le g(K)+c$.
The loss of $g(K)$ is due to the second inequality in~\eqref{eq:exc-beta}.
It can be avoided by using the proofs,
rather than the statements,
of Theorems~\ref{thm:main} (in the form of~\eqref{eq:main-1})
and~\ref{thm:set-g-window}.

\begin{theorem}\label{thm:struct-main}
  Let $g$ be as in Theorem~\ref{thm:main}.
  There is a positive constant $\cnew$\clabel{c:struct-main} such that,
  for all $k\in\N$,
  \begin{equation}\label{eq:struct-level}
    \m
    \left(
      K=k,
      \exc(g(k))>\cref{c:struct-main}
    \right)
    \le
    \cref{c:struct-main}
    2^{-g(k)}.
  \end{equation}
\end{theorem}

\begin{proof}
  As in the proof of Theorem~\ref{thm:main},
  we can assume that $g$ is integer-valued
  (for the excess, this uses the fact that $h_{\omega}$ is decreasing)
  and that $k\ge c_g$
  (smaller $k$ being absorbed by the constant $\cref{c:struct-main}$).
  Let $K(\omega)=k$
  and let $\omega$ belong to one of the blocks $\Omega$ used at the end of that proof,
  which have size $2^b$ with $b=k-g(k)+O(1)$
  and complexity $K(\Omega)\le g(k)$.
  Then $h_{\omega}(g(k))\le\log\left|\Omega\right|=b$,
  i.e., $\exc_{\omega}(g(k))\le b-k+g(k)=O(1)$;
  this replaces the estimate \eqref{eq:d} of $d(\omega\mid\Omega)$ in that proof.
  Therefore, the objects $\omega$ satisfying $K(\omega)=k$ and $\exc_{\omega}(g(k))>\cref{c:struct-main}$
  belong to the remainder,
  whose universal measure is less than $2^b2^{-k+O(1)}=2^{-g(k)+O(1)}$,
  provided $\cref{c:struct-main}$ is large enough.
  This proves~\eqref{eq:struct-level}.
\end{proof}

By Lemma~\ref{lem:exc-beta},
the lower bounds in Corollary~\ref{cor:set-g}, \eqref{eq:set-g-bis}, \eqref{eq:set-1}--\eqref{eq:set-4},
and~\eqref{eq:window}
do not change if $\beta$ is replaced by $\exc$,
and the upper bounds follow from Theorem~\ref{thm:struct-main}
in the same way as they follow from Theorem~\ref{thm:main}.
Therefore, Corollary~\ref{cor:set-g}, \eqref{eq:set-g-bis}, \eqref{eq:set-1}--\eqref{eq:set-4},
and~\eqref{eq:window}
remain true if $\beta$ is replaced by $\exc$.
Theorem~\ref{thm:struct-main} also allows us
to improve Corollary~\ref{cor:struct-immediate}.

\begin{corollary}\label{cor:struct-window}
  Let $g$ be as in Corollary~\ref{cor:set-g}.
  There are positive constants $c$ and $d$ such that,
  from some $a$ on,
  \begin{multline}\label{eq:struct-window}
    2^{-g(a)-c}
    \le
    \m
    \left(
      \exc(g(K))=\infty
      \mid
      a<K\le a+d
    \right)\\
    \le
    \m
    \left(
      \exc(g(K))>c
      \mid
      a<K\le a+d
    \right)
    \le
    2^{-g(a)+K(a)+c}.
  \end{multline}
\end{corollary}

\begin{proof}
  Let $d$ be as in the proof of Theorem~\ref{thm:set-g-window};
  in particular, $d\ge d_0$.
  For the upper bound,
  notice that, by~\eqref{eq:struct-level} and since $g$ is increasing,
  \[
    \m
    \left(
      a<K\le a+d,
      \exc(g(K))>\cref{c:struct-main}
    \right)
    \le
    \cref{c:struct-main}
    d
    2^{-g(a)},
  \]
  while $\m(a<K\le a+d)\ge2^{-K(a)-\cref{c:window}}$ for large $a$
  by Lemma~\ref{lem:window}.
  Dividing the former by the latter gives the upper bound in~\eqref{eq:struct-window}.

  The lower bound follows immediately
  from Theorem~\ref{thm:set-g-window} and Lemma~\ref{lem:exc-beta}.
\end{proof}

Specializing Corollary~\ref{cor:struct-window}
to the functions $g$ used in~\eqref{eq:set-1}--\eqref{eq:set-4},
we obtain analogues of those inequalities
for the excess and the conditional universal measure.
For a computable number $\theta\in(0,1)$ and $g(x):=\theta x$,
the conditional universal measures in~\eqref{eq:struct-window}
have the main term $2^{-\theta a}$,
which shrinks to zero exponentially fast as $a$ increases.
For $g(x):=x^{\theta}$,
the conditional universal measures have the main term $2^{-a^{\theta}}$,
the decay becoming stretched exponential.
Now let $\theta$ be a computable number in the range $\theta>1$.
For $g(x):=\log^{\theta}x$,
the main term becomes $2^{-\log^{\theta}a}$
and its rate of decay becomes superpolynomial.
For $g(x):=\theta\log x$, the rate of decay becomes polynomial.

Apart from Corollary~\ref{cor:struct-immediate},
the results of this appendix imply the corresponding results
of the main part of the paper
for the best fit function with finite sets as models.
Indeed, by Lemma~\ref{lem:exc-beta},
the events $\beta(g(K))=\infty$ and $\exc(g(K))=\infty$ coincide,
and $\beta(g(K))>c+\cref{c:exc}$ implies $\exc(g(K))>c$.
Therefore, Theorem~\ref{thm:struct-main} implies Theorem~\ref{thm:main},
the versions of Corollary~\ref{cor:set-g}, \eqref{eq:set-g-bis}, \eqref{eq:set-1}--\eqref{eq:set-4},
and~\eqref{eq:window}
for the excess imply the original statements,
and Corollary~\ref{cor:struct-window} implies Theorem~\ref{thm:set-g-window}
in a more precise form.
Theorem~\ref{thm:opt} is not affected:
by Lemma~\ref{lem:exc-beta}, its version for the excess coincides with the original.
Lemma~\ref{lem:exc-beta} does not allow us to go in the opposite direction
without losing $g(K)$ in the threshold,
as in Corollary~\ref{cor:struct-immediate},
and in this sense the results for structure functions are stronger.
\end{document}